\documentclass{amsart}
\usepackage{amscd,amsfonts,mathrsfs,amsthm,amssymb}
\usepackage{fancyhdr}
\usepackage[latin1]{inputenc}

\usepackage{color}
\usepackage{amsthm,amssymb,verbatim}
\usepackage{graphicx}
\usepackage{enumerate}

\usepackage{chngcntr}
\usepackage{color}
\usepackage{amsthm,amssymb,verbatim}
\usepackage{graphicx}
\usepackage{enumerate}
\usepackage{hyperref}
\usepackage[alphabetic, msc-links, backrefs]{amsrefs}
\usepackage[latin1]{inputenc}

\newcommand{\Aa}{\mathcal A} 

\newcommand{\curves}{\mathfrak{C}}

\newcommand{\Ff}{\mathcal F}
\newcommand{\Rr}{\mathcal R}
 \newcommand{\Cc}{\mathcal C}

 \newcommand{\RR}{\mathbf{R}}  
 
 \newcommand{\BB}{\mathbf{B}}  

    \newcommand{\dist}{\operatorname{dist}}
 
 \newcommand{\area}{\operatorname{area}}
 \newcommand{\eps}{\epsilon}

\newcommand{\vv}{\mathbf v}
\newcommand{\ee}{\mathbf e}

\usepackage{amsthm}

\newcommand{\zmax}{z^*}
 
\newcommand{\pdf}[2]{\frac{\partial #1}{\partial #2}}

 \newcommand{\partialin}{\partial_\textnormal{inner}}
\newcommand{\partialout}{\partial_\textnormal{outer}}
\newtheorem*{*theorem}{Theorem}
\input epsf
\def\begfig {
\begin{figure}
\small }
\def\endfig {
\normalsize
\end{figure}
}

    \newtheorem{theorem}    {Theorem}       [section]
    \newtheorem{lemma}      [theorem]       {Lemma}
    \newtheorem{corollary}  [theorem]     {Corollary}
    \newtheorem{proposition}       [theorem]       {Proposition}
    
    \newtheorem{claim}{Claim}
    \newtheorem*{claim*}{Claim}
    \newtheorem*{theorem*}{Theorem}
    \theoremstyle{definition}
    \newtheorem{definition}  [theorem] {Definition}
     \newtheorem{conjecture}  [theorem] {Conjecture}
    \theoremstyle{definition}
    \newtheorem{remark}   [theorem]       {Remark}

\begin{document}
\title[Annuloids and $\Delta$-wings]{Annuloids and $\Delta$-wings}
\author[D. Hoffman]{\textsc{D. Hoffman}}

\address{David Hoffman\newline
 Department of Mathematics\newline
 Stanford University \newline
   Stanford, CA 94305, USA\newline
{\sl E-mail address:} {\bf dhoffman@stanford.edu}}

\author[F. Martin]{\textsc{F. Mart\'in}}

\address{Francisco Mart\'in\newline
Departmento de Geometr\'ia y Topolog\'ia  \newline
Instituto de Matem\'aticas IMAG Granada \newline
Universidad de Granada\newline
18071 Granada, Spain\newline
{\sl E-mail address:} {\bf fmartin@ugr.es}
}
\author[B. White]{\textsc{B. White}}

\address{Brian White\newline
Department of Mathematics \newline
 Stanford University \newline 
  Stanford, CA 94305, USA\newline
{\sl E-mail address:} {\bf bcwhite@stanford.edu}
}

\date{23 August, 2023}
\thanks{F. Martín  was partially supported by the MICINN grant PID2020-116126-I00, by the IMAG--Maria de Maeztu grant CEX2020-001105-M / AEI / 10.13039/501100011033 and  by the Regional Government of Andalusia and ERDEF grant P20-01391.
B. White was partially supported by grants from the Simons Foundation
(\#396369) and from the National Science Foundation (DMS~1404282, DMS~ 1711293).}


\begin{abstract}
We describe new annular examples of complete translating solitons for the mean curvature flow and how they are related to a family of translating graphs, the $\Delta$-wings. In addition, we will prove several related results that answer questions that arise naturally in this investigation. These results apply to translators in general, not just to graphs or annuli.
\end{abstract}

\maketitle

\section{Introduction}\label{sec:intro}

A {\bf translator} in $\RR^3$ is a surface $M$ such that 
\[
   t\mapsto M- t \,\ee_{3}
\]
is a mean curvature flow, i.e., such that normal component of the velocity at each point is
equal to the mean curvature at that point:
\begin{equation}\label{general-translator-equation}
   \overrightarrow{H} = -\ee_{3}^\perp.
\end{equation}
As observed by Ilmanen~\cite{ilmanen_1994},  a surface $M\subset\RR^3$ 
is a translator if and only if
it is minimal with respect to the Riemannian metric

\[
    g_{ij} = e^{-z} \delta_{ij}.
\]
 
In \cite{graphs},  we classified all the translators that are graphs over domains in $\RR^2=\{z=0\}$. 
(See also the survey paper \cite{himw-survey}.) That classification depends on the fundamental advances in the paper of Spruck and Xiao \cite{spruck-xiao}. What we learned from that classification led us to a construction of complete annular translators.\cite{annuloids} 
In this paper,  we will describe  these new annular examples  and how they  are related to a family of graphs, the $\Delta$-wings. 

In addition we present three  related results  that answer questions that arise naturally in this investigation. They apply to translators in general, not just to graphs or annuli.  The first one is proved in \cite{annuloids} and depends upon a basic result (Proposition \ref{prop:gap}) that is important to the understanding of  the other two theorems.

\begin{theorem}
[\cite{annuloids}*{Theorem~18.5}] 
\label{th:gap3}
Let $U_n\subset U_n'$ be nested, open, convex regions in $\RR^2$ such that $U_n$ converges to a bounded open 
convex set $U$ and such that $U_n'$ converges to an infinite strip $U'$.
Suppose that
\[
    \min\{ |p-q|: p\in \partial U, q\in \partial U'\} \ge \pi.
\]
Then, for all sufficiently large $n$, there is no connected translator in $\{z\ge 0\}$ whose
boundary is $S_n:=((\partial U_n)\cup (\partial U_n'))\times\{0\}$.
\end{theorem}

\begin{remark}\label{rem: R3-gap}Theorem~\ref{th:gap3} is a reminiscent of a  classical result  for minimal surfaces in $\RR^3$, a version of which could be stated as follows:  For $C_0$ a closed convex curve in 
$\RR^2=\{z=0\}\subset\RR^3$, let $C_t=C_0+t\ee_3$. Then for $t>0$ sufficiently large, 
$C_0\cup C_t $ bounds no connected minimal surface.
\end{remark}

\begin{theorem} \label{th:bounded-by-lines} Let $M$ be a connected translator in $\{z\geq 0\}$ that lies in the slab $\{|y|<B\}$ and has boundary equal to the
two parallel lines $\{y=\pm b\}\cap \{z=0\}$.  Suppose that either there exists a value of $c\in \RR$ for which
$M\cap \{x=c\}$ is bounded, or that $M$ is simply connected.
Then $b<\pi/2$, and $M$ is part of the the graph of an appropriately translated grim reaper surface:
$$ z=\log(\cos y)-\log(\cos b) \mbox{ on the strip } \{(x,y): |y|<b\}.$$
\end{theorem}
See Section~\ref{sec:graphs} for a discussion of grim reaper graphs.
If one assumes that $M$ is a graph, this theorem  is well known.
 We are not assuming that here.

\begin{theorem}\label{th:narrow-slabs} Suppose $M$ is a properly embedded and connected translator that lies in a vertical slab $\{|y|<B\}$.
If there exists a  constant $c$ such that $M\cap\{x=c\}$ is bounded above, then $B\geq \pi$.
\end{theorem}

\begin{remark}Recently,  Gama, Mart\'in and M\o ller have proved a related result: 

\begin{theorem}[\cite{GMM22}* {Proposition 9.1}] Let $M$ be a complete, embedded, connected translator with finite genus, finite entropy and one end. Suppose
$M$ lies in a slab of width $B$. Then $B \geq\pi$. 

\end{theorem}
\noindent We  were recently  informed that in  \cite{DE-NM-MR}, D. Impera, N. M\o ller and M. Rimoldi  obtain, by different methods,   results related to Theorem~\ref{th:narrow-slabs} for  complete translators  in a slab. They assume  finite entropy, and  height  satisfying a
 linear-growth condition  dependent on the width of the slab. 

\end{remark}
\begin{remark} In Theorem~\ref{th:narrow-slabs}, we do not know whether the assumption that $M\cap\{x=c\}$ is bounded above for some $c$ can be removed.  Imagine
a connected annular surface that is asymptotic to two  parallel vertical planes, and looks like two planes connected by a small, catenoid-like, neck. If the planes are at a distance  $b$ from each other, then a maximum-principle argument using the grim reaper surface will show that $b<\pi$. However, we do not know at the present time how to rule out the existence of such a surface when the planes are very close together (see Fig. \ref{fig:non}.) Such a surface, if it exists, would lie in a slab of width less than $\pi$ and  
$M\cap\{x=c\}$ would be unbounded for any value of $c$.

\end{remark}
 \begin{figure}[htbp]
\begin{center}
\includegraphics[height=4.5cm]{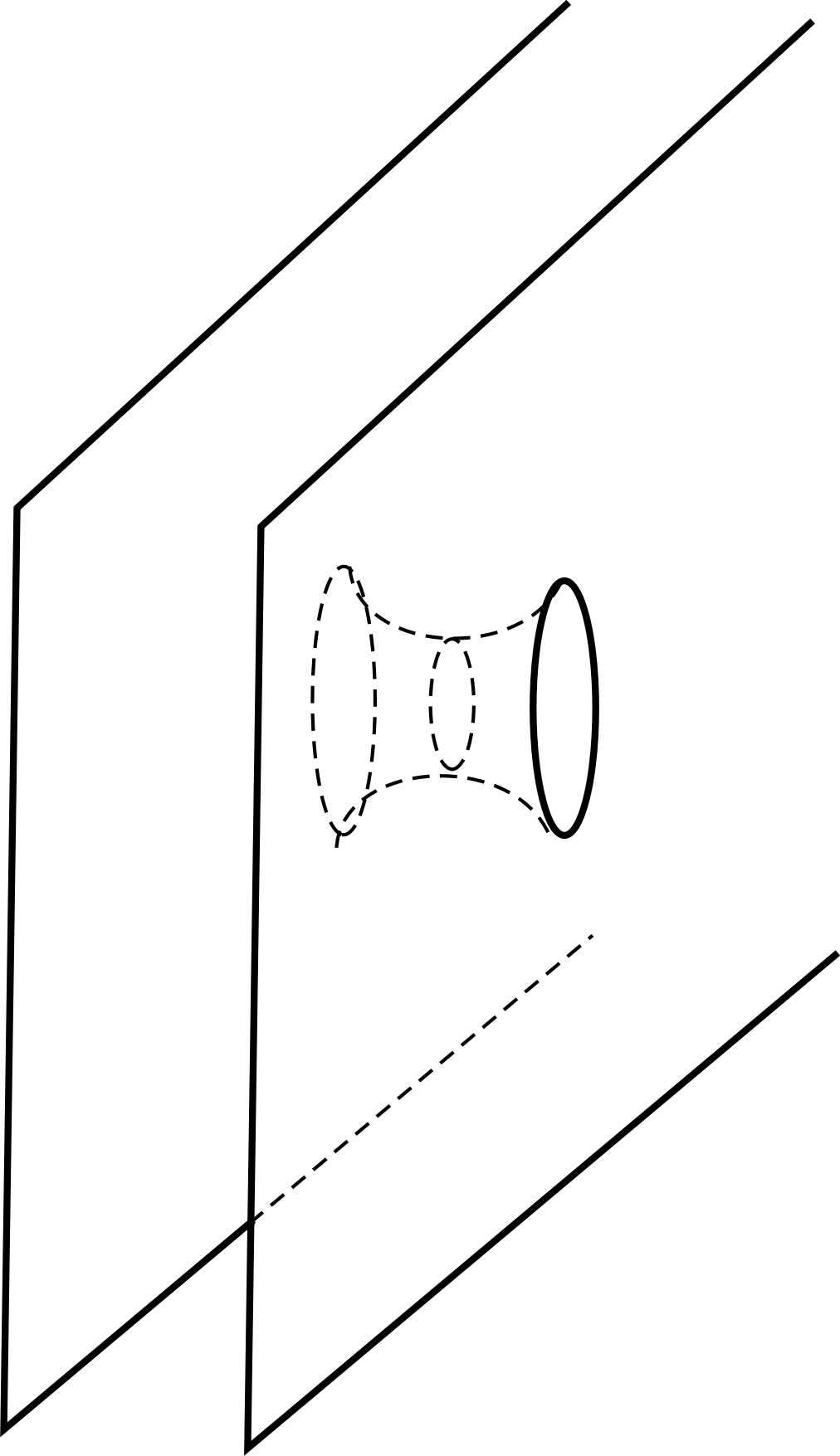}
\caption{\small Two vertical planes glued by a small catenoidal neck. We conjecture that such a translator does not exist.}
\label{fig:non}
\end{center}
\end{figure}

  The paper is organized as follows. In Section~\ref{sec:annuloids}, we define annuloids and state the main existence theorem for the
    family $\Aa$ of annuloids that we construct as limits of finite surfaces with boundary. In Sections~\ref{sec:graphs} and \ref{sec:DeltaWings}, we review the theory of complete translating graphs with special emphasis on $\Delta$-wings  as models for the construction of the annuloids $\Aa$. In Section~\ref{sec:compact-annuli} we outline the proof of the existence of the compact annular surfaces with boundary whose limits in Section~\ref{sec:Aa} are the complete annular translators in $\Aa$.  In Section~\ref{sec:proofs}, we provide the proofs of Theorems~ \ref{th:gap3}, \ref{th:bounded-by-lines}, and \ref{th:narrow-slabs}, above.

 \section{Annuloids}\label{sec:annuloids}
 In \cite{annuloids}
we  construct a two-parameter family of 
complete embedded annular surfaces that we call annuloids. 
\begin{definition} \label{def: annuloid} An {\bf annuloid}  is  a properly embedded translator $M$ such that
\begin{enumerate}
\item $M$ is an annulus.
\item $M$ lies in a slab $\{|y|\le B'\}$.
\item $M$ is symmetric with respect to reflection in the vertical coordinate planes.
\item $M+(0,0,z)$ converges as $z\to \infty$ to four planes $\{y=\pm b\}$ and $\{y=\pm B\}$ for some $0<b\le B$.
\item $M - (0,0,z)$ converges as $z\to\infty$ to the empty set.
\item $M$ is disjoint from the $z$-axis $Z$.

\end{enumerate}
We define the {\bf width} of $M$ to be the number $B=B(M)$. (One can prove that $B$ is also the smallest $B'$ such that $(2)$ holds.)   We define the {\bf inner width} of $M$ to be the number $b=b(M)$.\end{definition}

To state the main theorem below precisely, one needs to specify a notion of necksize of an annulus. There are various natural definitions, such as:
 the length of the shortest homotopically nontrivial curve in $M$;
 the radius of the smallest ball containing a nontrivial curve in $M$; the radius of the smallest vertical cylinder
containing a nontrivial curve in $M$.  Our existence result is true for any of those definitions.
However, the following turns out to be most convenient notion of necksize:

\begin{definition}\label{x(M)}
If $M$ is a surface, we let $x(M)$ be the distance from  the $Z$ axis to $M\cap \{y=0\}$. We refer to $x(M)$ as the {\bf necksize of $M$}.
\end{definition}

In \cite{annuloids}, we prove the existence of a collection of annuloids that behaves like a two-parameter family:
\begin{theorem}\label{annuloids Aa}
There exists a family $\Aa$ of annuloids with the property that the map 
$$\Aa\rightarrow [\pi/2,\infty)\times(0,\infty)$$
given by $M\mapsto (b(M), x(M))$ is continuous, proper, and subjective.  Furthermore, for fixed $b\geq\pi/2$, 
$$\Aa_b:=\{M\in \Aa\,:\, b(M)=b\}$$ contains a closed and connected subset   $\Cc_b$ on which the map 
\begin{align*}
\Cc_b&\rightarrow (0,\infty)\\
M &\mapsto x(M)
\end{align*}
 is continuous and surjective. 
\end{theorem}

\noindent The family  $\Aa$ is defined in Definition~\ref{def:Aa}, which refers to Proposition~\ref{prop:Mbx} and Theorem~\ref{th:a,b,x}.
\begin{corollary}\label{corollary-1}
For each $b\ge \pi/2$ and for each $0<s< \infty$,
there exists an annuloid in $\Aa$ with inner width $b$ and necksize $s$.
\end{corollary}
\begin{figure}[htbp]
\begin{center}
\includegraphics[height=6.5cm]{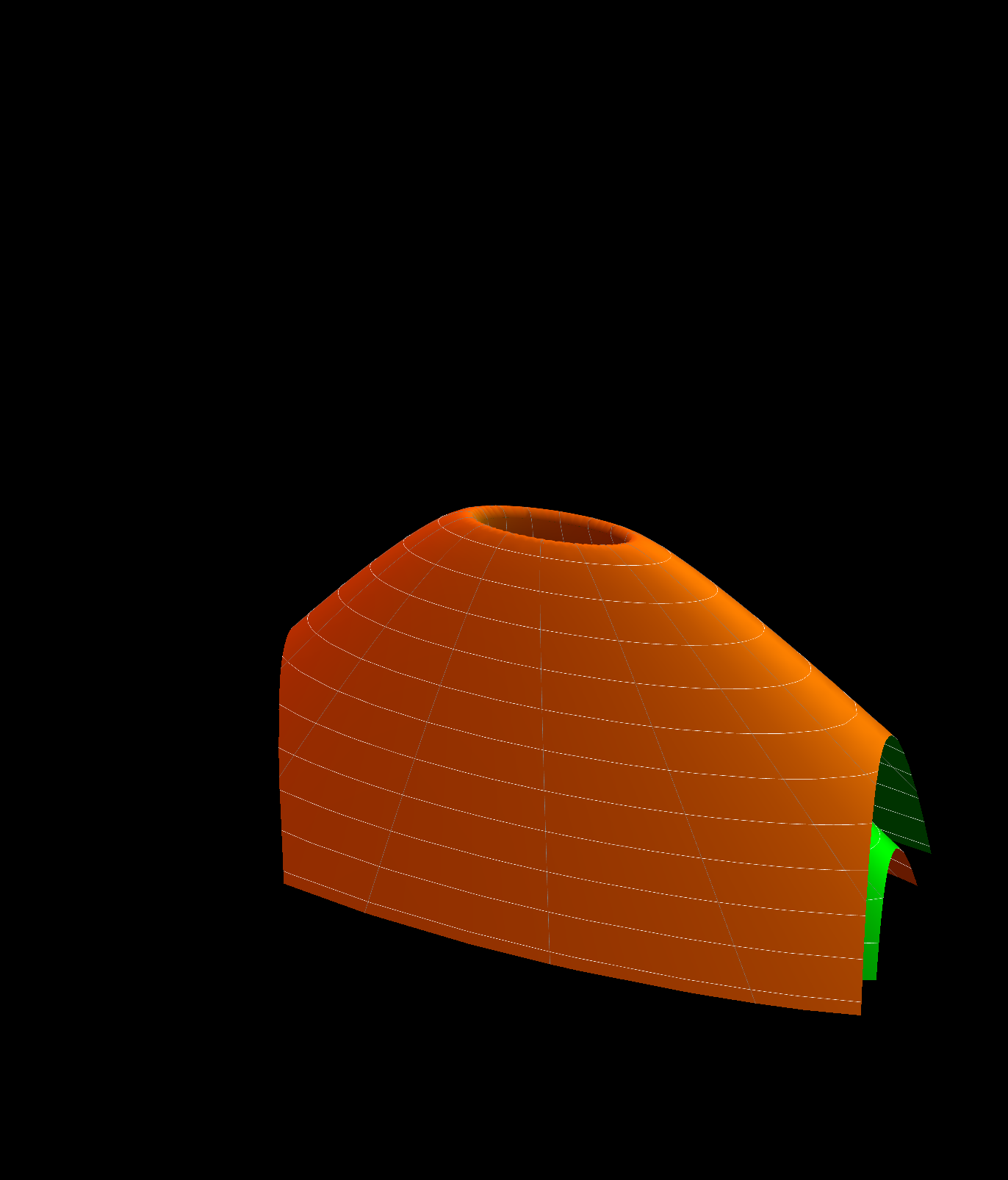}\includegraphics[height=6.5cm]{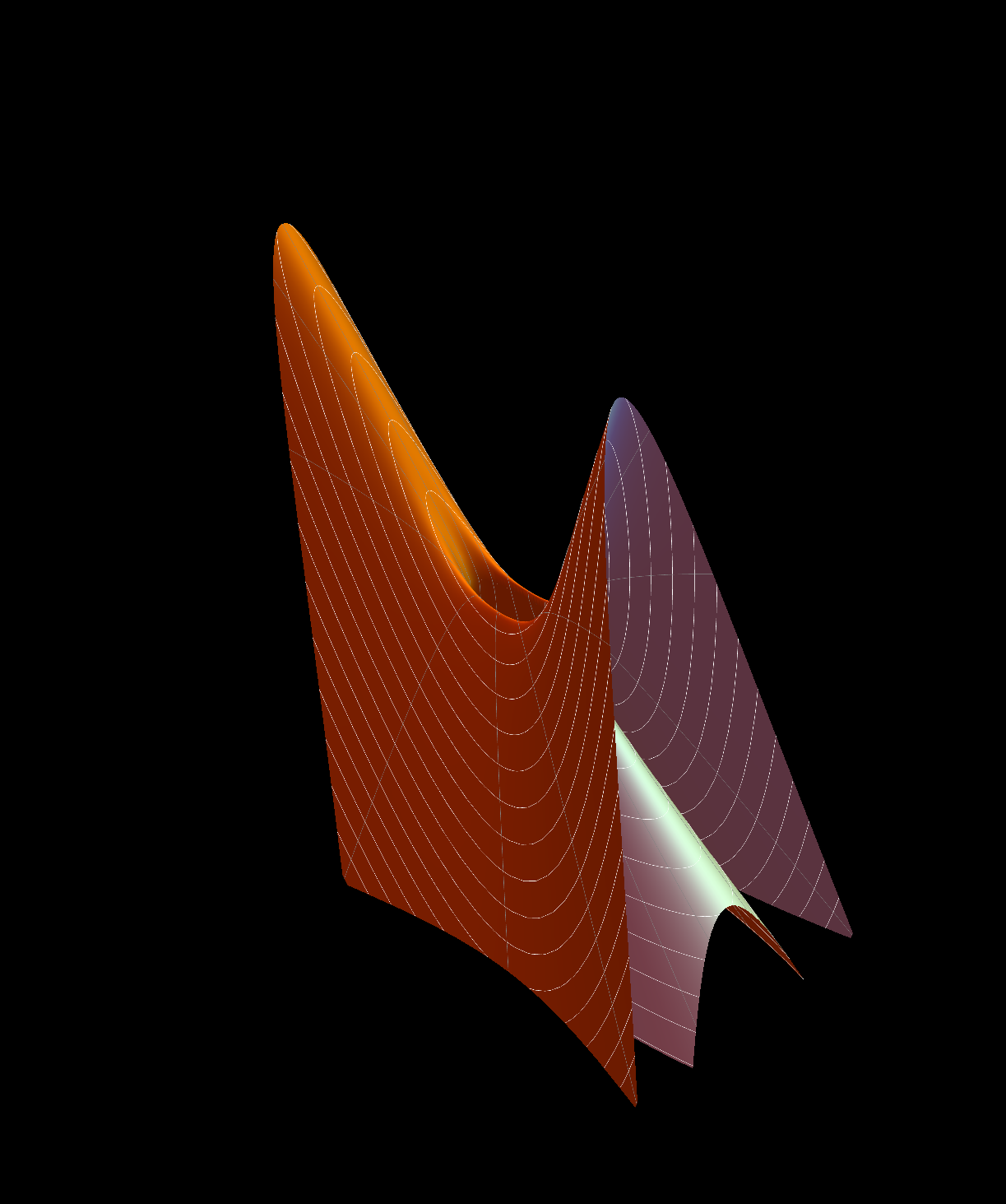}
\caption{\small A {\bf capped} annular translator in $\Aa$ (left) and an {\bf uncapped} one (right). When $B>b$ (see 
  Definition~\ref{def: annuloid}), the annuloid is uncapped.
When $B=b$, the annuloid may be capped or uncapped. The transition from capped to uncapped is addressed in \cite{annuloids}.}
\label{fig:capped-uncapped}
\end{center}
\end{figure}

Our discovery of the annuloids in Theorem~\ref{annuloids Aa} was guided by our construction of $\Delta$-wings
in  \cite{graphs}  as limits of 
translates of graphs over finite rectangles in the plane.  We produce annuloids as limits of 
 compact, annular translators whose  boundaries consist of pairs of symmetrically placed nested rectangles in a horizontal plane. Our existence and uniqueness proof for 
 $\Delta$-wings relies heavily on  the advances made in the paper  by Spruck and Xiao~\cite{spruck-xiao}, and the techniques there appear in parts of our construction of annuloids.

To our knowledge, the only properly embedded, annular translators that were known before \cite{annuloids} are the rotationally invariant surfaces called {\bf translating catenoids} (Figure~\ref{fig:catenoid}), discovered by 
Clutterbuck, Schn\"urer and Schulze \cite{CSS} (see also \cite{Altschuler-Wu}). They showed that there is a one-parameter family $\{W(\lambda)\}_{\lambda>0}$ of such surfaces.   The parameter $\lambda$ 
is the radius of the neck circle. It coincides with our definition of necksize  for annuloids. They also examined the limit as  $\lambda\rightarrow 0$ of $W(\lambda)$. It consists of two superimposed copies of an entire, rotationally invariant graph known as the bowl soliton. (See Section~\ref{sec:graphs}  and Figure~\ref{fig:catenoid}.) The convergence is smooth away from the point on the axis of symmetry where the neck collapses. 

As part of our investigation of  the annuloids in $\Aa$, we were able to show that if $b$ is fixed and $s$ goes to zero, the associated annuloids in Corollary~\ref{corollary-1} converge  (with multiplicity two) to the $\Delta$-wing defined over the strip $\RR\times (-b,b)$. 
(See Figure~\ref{fig:DeltaWing}.) The convergence is smooth away from the point where the $Z$ axis intersects the  $\Delta$-wing. We conjecture the following behavior as $b\rightarrow \infty$:

\begin{conjecture} Fix $s$  and let $M_i$ be an annuloid in $\Aa$, as in Corollary~\ref{corollary-1} with  necksize  $s$ and inner width $b_i\rightarrow\infty$. Then (after suitable vertical translations) the $M_i$ converge to the rotationally symmetric translating catenoid $W(s)$, whose neck is a circle of radius $s$.

\end{conjecture}

\section{Translating Graphs}\label{sec:graphs}
The simplest complete translators  are vertical planes. In \cite{CSS}, J. Clutterbuck, O. Schn\"urer and F. Schulze  (see also \cite{Altschuler-Wu}) 
proved that there is a unique (up to vertical translation) entire, rotationally invariant
function $u: \RR^2\to \RR$ whose graph is a translator.
It is called the {\bf bowl soliton}~(Fig.~\ref{fig:catenoid}).

The simplest, non-entire, complete graph over a region in $\{z=0\}$ is the cylinder constructed over the grim reaper curve, 
$$u_{\pi/2}(x,y)=\log (\cos y),$$
 $y \in (-\pi/2,\pi/2)$, $x\in\RR$.  We refer to this graph as the {\bf grim reaper surface}. It  can be tilted and dilated to obtain  a complete translating graph over a strip of width $2 b \geq \pi$:

 $$u_b:\RR\times (-b,b)\rightarrow \RR$$
\begin{equation}\label{TGR2}
  (x,y) \mapsto \left(\frac{2 b}{\pi}\right)^2 \log \left(\cos \left(\frac{y \pi}{2 b}\right)\right) + x \tan(\theta),   
  \end{equation}
where $\tan(\theta)=\sqrt{(2 b/\pi)^2-1}$. 
We call this graph a {\bf tilted grim reaper surface}.
(The slope of the graph of  $u_b$ is  $\tan(\theta)$. Of course, the graph of $u_b(-x,y)$ is also a translator. We refer to this graph as being {\em negatively tilted}.)

\begin{remark}\label{rmk:flat-implies GR} A tilted grim reaper surface $M$ is a cylinder, so its Gauss curvature is identically equal to zero.  Along a straight line on the surface, the Gauss map is constant: the Gaussian image of $M$  is a half circle in the upper hemisphere. The half circles
corresponding to the tilted grim reaper surfaces foliate the upper hemisphere. (We include the grim reaper surface in this collection as a tilted grim reaper surface with tilt angle $0$.) By Massey's theorem, \cite{Massey}
a complete surface in $\RR^3$ whose Gauss curvature is identically $0$ is a cylinder. Given $\Sigma$, a complete translating graph with Gaussian curvature identically equal to $0$, if $L$ is a line on $\Sigma$ we can find a tilted grim reaper surface $M$ such that (after suitable translation and rotation) $M$ and $\Sigma$ are tangent along $L$. By Cauchy-Kowalevski,  $\Sigma=M$.  So, up to translation and rotation the complete, flat, translating graphs are tilted grim reaper surfaces. In particular, there are no complete, flat, translating graphs defined over  strips of width less than $\pi$.
\end{remark}

 In \cite{spruck-xiao}, Spruck and Xiao proved that a  complete  translating graph
has nonnegative Gauss curvature: $K\geq 0$. But since
$k_1/H$ satisfies a strong maximum principle on translators \cite{white-nature} it follows that if K=0 anywhere on a translating graph then  $K\equiv 0$. (Here, $0\leq k_1\leq k_2$ are the principal curvatures and $H=k_1+k_2$.)  
Therefore, in order to classify complete translating graphs, it suffices to classify the complete translating graphs with positive Gauss curvature.
 The bowl soliton is not a cylinder so its Gauss curvature must be strictly positive. (This also follows directly from a computation.)  Do other examples exist?
Spruck and Xiao (\cite{spruck-xiao}*{Theorem~1.5})
showed that such surfaces, if they are graphs over strips, are reasonably well behaved.
\begin{proposition} \label{strictly convex graph} A complete translating graph
$$ u:\RR\times (-b,b)\rightarrow \RR$$
with $K>0$ satisfies $u(x,-y)=u(x,y)$. Furthermore,
$$u(x+t,y)-u(t,0)$$ converges smoothly as $t\rightarrow-\infty$, to a tilted grim reaper  surface \eqref{TGR2} defined over the strip of width $2b$ and, as $t\rightarrow \infty$, to the negatively-tilted grim reaper surface of the same width. In particular, $b>\pi/2$. 
\end{proposition}

\begin{remark}\label{remark: no-thin-graphs}
 From  the discussion above, we know  that there are no complete, flat, translating graphs defined over strips of width less than $\pi$.   From Proposition~\ref{strictly convex graph} it follows that there are no complete, translating graphs of any kind defined over such thin strips.
\end{remark}

What are these surfaces? 
Ilmanen, in an unpublished work, described a new family of translating graphs called $\Delta$-wings:
they are defined over strips, and they have strictly positive Gauss curvature at some points.
Hence, as argued in  the paragraph before Proposition~\ref{strictly convex graph},
  they have positive Gauss curvature at all points.
  With Ilmanen, we gave an existence proof for these surfaces and  proved the following classification theorem for graphical solitons:
 
\begin{theorem}[\cite{graphs}] \label{th:graphs}Up to isometries of $\RR^2$ and vertical translation, the only complete translating graphs in $\RR^3$ are the grim reaper surface, the tilted grim reaper surfaces, the $\Delta$-wings, and the bowl soliton.
\end{theorem}
 \begin{figure}[htbp]\label{fig:DeltaWing}
\begin{center}
\includegraphics[height=4.5cm]{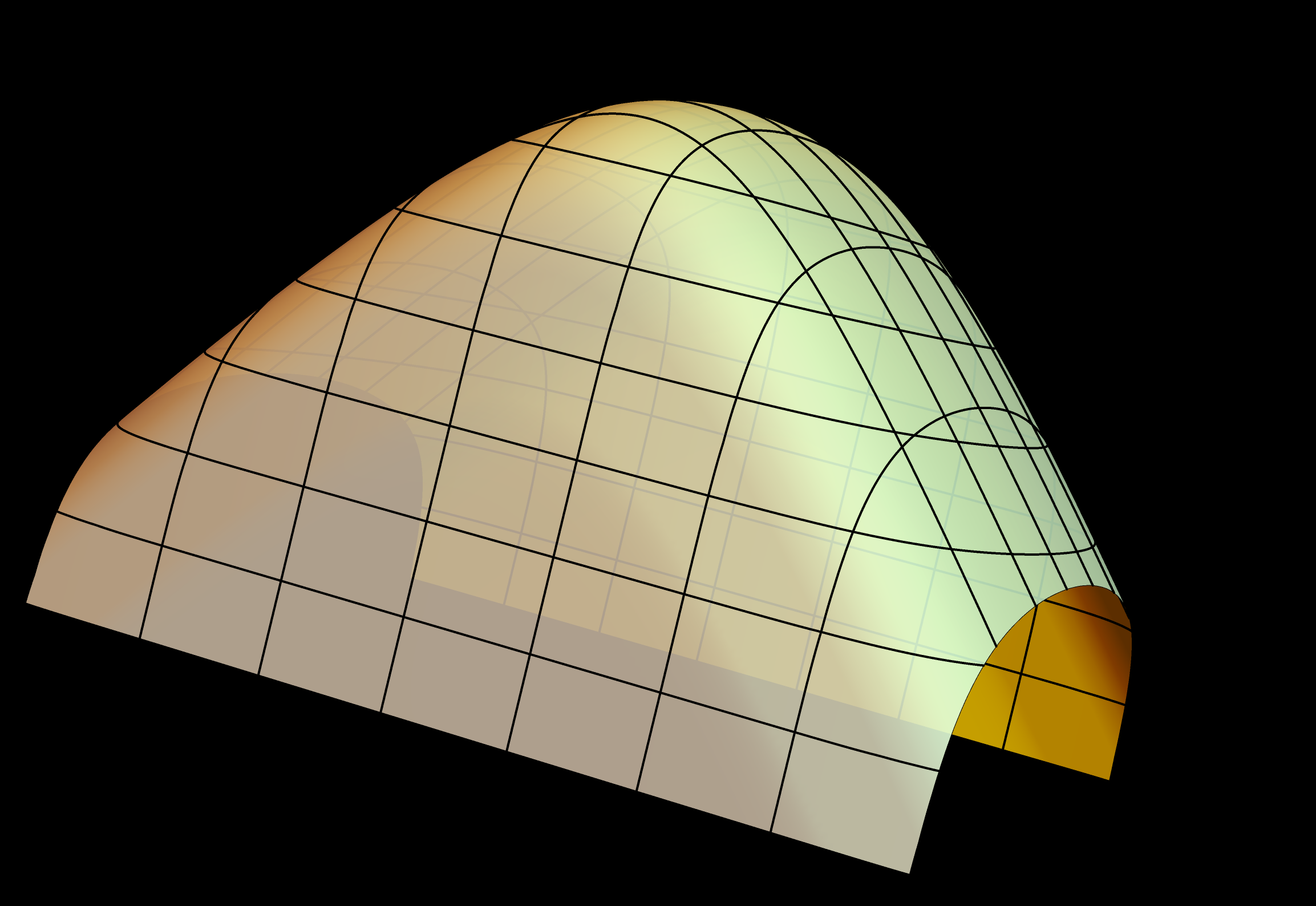}
\caption{\small A $\Delta$-wing.}
\label{fig:delta}
\end{center}
\end{figure}

In the next section, we will outline our construction  of $\Delta$-wings in order to set the stage for the  construction of the annular translators described in Sections~\ref{sec:compact-annuli} and \ref{sec:Aa}.
Here is what we know about them at this point.
\begin{enumerate}
\item They must have strictly positive curvature.
\item They are symmetric with respect to reflection in the coordinate plane $\{y=0\}$.
\item They are asymptotic to positively tilted grim reaper surfaces as $x\rightarrow -\infty$ and to negatively tilted grim reaper surfaces as  $x\rightarrow +\infty$.
Hence, they are bounded above.
\noindent

\end{enumerate}

\section{$\Delta$-wings: complete translating graphs that are limits of graphs over rectangles}\label{sec:DeltaWings}

Consider a rectangle $R$ in the  horizontal plane $\{z=0\}$. Define $D(R)$ to be the translating graph over $R$ with zero boundary values. 
(As observed in Section~\ref{sec:intro}, a surface $M\subset\RR^{3}$ 
is a translator if and only if
it is minimal with respect to the Riemannian metric $g_{ij} =  e^{-z}\delta_{ij}$.)

 For $0\leq t \leq 1$, let $g(t)$ be the metric $g(t)_{ij}=e^{-tz } \delta_{ij}$.  We may use the continuity method to find a 
graph that is a  $g(t)$-minimal surface with $0$ boundary values.  (If $t=0$, $g(0)$ is the standard euclidean metric, then $R$ itself is the required graph. The solutions are bounded below (by $0$) and above by a  bowl soliton.) Vertical translations of the $g$-minimal graph foliate $R\times \RR$. It then follows from the maximum principle for minimal surfaces that $D(R)$ is the unique graphical solution (indeed the unique compact $g$-minimal surface) with the same boundary as $D(R)$.

Let $R_{L,b}$ be the rectangle $[-L, L] \times [-b,b]\subset \{z=0\}\subset\RR^2$. It follows from uniqueness that $D(R_{L,b})$ is symmetric with respect to reflection in the coordinate planes $\{x=0\}$ and $\{y=0\}$. Denote by
\begin{equation}  \label{eqn:D(L,b)}
 u_{L,b} :    [ -L, L] \times [-b,b] \longrightarrow \RR
\end{equation}
 the function with  $0$ boundary values whose graph  is $D(R_{L,b})$.

{\bf Limits of the  disks $D(R_{L,b})$  as $L\rightarrow \infty$.} We want to produce complete graphical solitons by fixing $b>0$ and taking limits as $L\rightarrow\infty$ of the vertically-translated surfaces $$D(R_{L,b})-(0,0,u_{L,b}(0,0)).$$ 
 
 It is not hard to show that the maximum value of $u_{L,b}$ is achieved at $(0,0)$. Therefore these surfaces are bounded above by $0$ and contain the origin.  In order  to insure that one gets (subsequential) limits, we prove a gradient estimate 
for $u_{L,b}$:  the norm $|Du_{L,b}|$ is bounded, independent of $L$,  on compact subsets of $\RR\times(-b,b)$. Therefore subsequential limits exist and give examples of graphical translators passing through the origin and lying in $\{z\leq0\}$.

 Let $\Sigma$ be a limit translator. The surface $\Sigma$ is not complete unless $u_{L,b}(0,0)\rightarrow \infty$ in this subsequence. (By the maximum principle, if $L'>L$, then $u_{L',b}(x,y)> u_{L,b}(x,y)$ on the interior of $R_{L,b}$. Therefore  $u_{L,b}(0,0)$ is  a monotonically increasing as a function of $L$ and a limit exists (possibly infinite).) If, as $L\rightarrow\infty$, $u_{L,b}(0,0)$ is finite, the limit surface is bounded by parallel lines in a horizontal plane.  If  as $L\rightarrow\infty$, $u_{L,b}(0,0))\rightarrow \infty$, the limit surface is complete. In both cases, it is symmetric with respect to reflection in the vertical coordinate planes.
 \begin{figure}
 \begin{center}
\includegraphics[height=4.5cm]{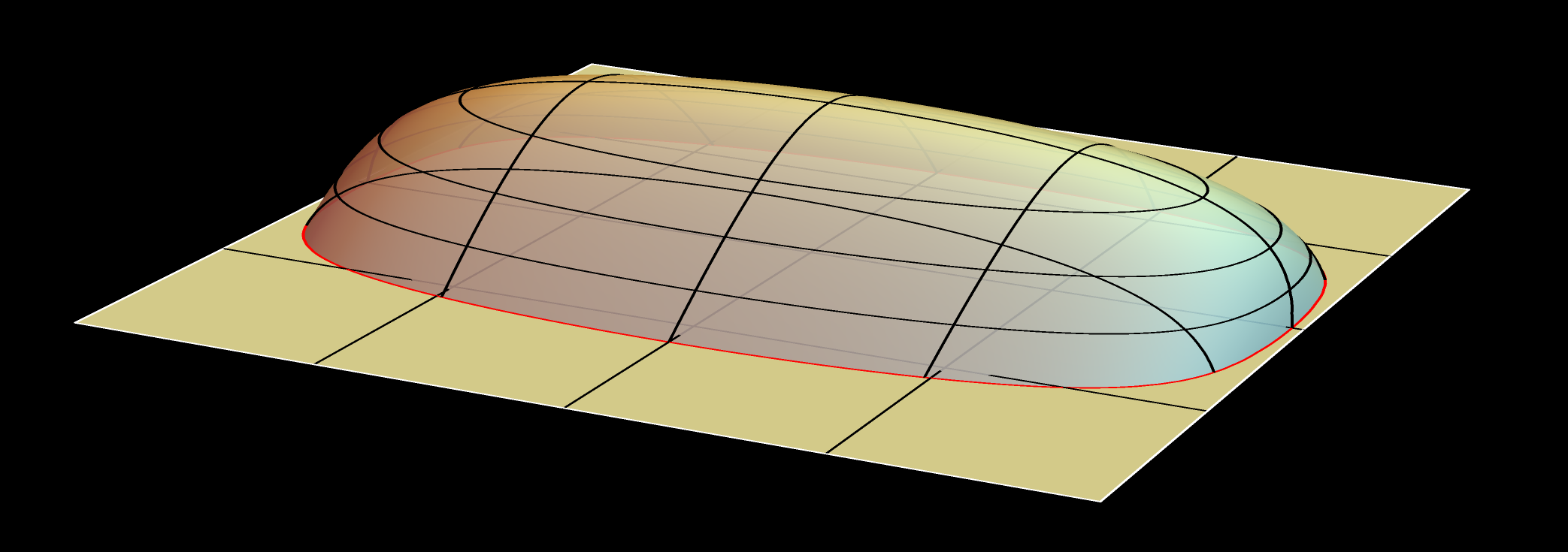}
\caption{\small A translating disk bounded by a convex planar curve.}
\label{fig:disk}
\end{center}
\end{figure}
 
\begin{proposition} [\cite{graphs}*{Corollary~3.3}]\label{prop:bounded-limit} Suppose $\lim_{L\rightarrow\infty}u_{L,b}(0,0)=C <\infty$. Then $b<\pi/2$. 

\end{proposition} 

\begin{proof}   Consider the functions

\begin{equation} \label{eqn:u-L-b}
u_{L,b}(x,y)-u_{L,b}(0,0):R_{L,b}\rightarrow\RR.
\end{equation}
Notice that the value of these functions is $0$  at $(0,0)$, where the tangent planes to their graphs are horizontal.
As $L\rightarrow\infty$, a subsequence converges to a graph of  a function  $$v_b:\RR \times(-b,b)\rightarrow \RR.$$
The function has boundary values $-C$ on the lines $\{z=-C\}\cap\{y=\pm b\}$. 
Applying Theorem~\ref{th:bounded-by-lines} to the vertical translation of the surface by $(0,0,C)$ completes the proof.
 \end{proof}
 \begin{remark}\label{rem:b<pi/2}
It follows from Theorem~\ref{th:bounded-by-lines} that 
$$\lim_{L\rightarrow\infty} u_{L,b} (x,y) =\log(\cos y) +C,$$ 
where $C=-\log(\cos b)$, when
$b< \pi/2$. It is not hard to prove that in fact, one does not need to take subsequential limits: all the limits are the same. 
 \end{remark}

What happens when $b\geq\pi/2\,$? the functions  $$u_{L,b}$$  in \eqref{eqn:u-L-b}  have subsequential limits that produce 
 complete translating graphs. When $b=\pi/2$, it is easy to show that in fact all the limits are the same: the grim reaper surface. When $b>\pi/2$ there are two important questions to be answered:
 \begin{enumerate}
\item If a subsequential limit 
 exists, is it defined over a strip of width $2b$, or possibly over thinner  strip 
 inside $\RR\times(-b,b)$ ?
\item Are there different subsequential limits?

 \end{enumerate}
 The answers are the best possible:  the limit function is defined over the full strip  $\RR\times(-b,b)$ and it is unique. The fact that the limit functions are defined over the full strip follows from a gradient estimate for $u_{L,b}$ that depends only on an upper bound for $b$ and a lower bound for $b-|y|$. Uniqueness is  more complicated to prove, but it is true. We arrive at the following result:
 
 \begin{theorem}(\cite{graphs})\label{th:DWing}
 Let $b>\pi/2$. Then, modulo translations, there is a unique complete translator $f_b:\RR\times(-b,b)\rightarrow \RR$ that is not a tilted grim reaper surface. 
 These surfaces, called $\Delta$-wings, are symmetric with respect to reflection in the vertical coordinate planes: $f_b(x,y)=f_b(-x,y)= f_b(x,-y)$.
 They form a smooth family. As $b\rightarrow\pi/2$ they converge to the grim reaper surface. As $b\rightarrow\infty$,  they converge to a bowl soliton.
 \end{theorem}

 \section{Compact translating annuli bounded by convex  curves}\label{sec:compact-annuli}
We produced $\Delta$-wings as the limits of vertical translates of the graphs $D(R_{L,b})$ for fixed $b\geq \pi$ and $L\rightarrow\infty$.
 (Recall that $R_{L,b} $  is the rectangle $[-L,L]\times[-b,b]$ in the plane $\{z=0\}$.) 
 The surfaces in Theorem~\ref{annuloids Aa} are produced by taking limits of  vertical translates  of compact translating annuli bounded by a pair of disjoint, nested,  symmetrically placed 
 rectangles in the plane $\{z=0\}$.  In this section and the next one,  we will describe how we produce such annuli with  desired neck size.

 \begin{definition}\label{def:Acomp} We define  $\Cc$ to be  the  space of compact, properly embedded translating annuli $M$  in  the upper halfspace 
$\RR^2\times [0,\infty)$ such that:
\begin{enumerate}[1)]
\item   $\partial M$ is a pair of disjoint, nested, $C^{2,\alpha}$ convex curves with strictly positive curvature  in $\RR^2\times\{0\}$.
\item  $M$ is invariant under reflection in the planes $\{x=0\}$ and $\{y=0\}$.
\end{enumerate}

 Note that if $M \in \Cc$, then the curves in $\partial M$ will  also  be invariant under reflection in the planes $\{x=0\}$ and $\{y=0\}$.
  We denote  these boundary curves by
   $\partialin M$ and $\partialout M$. (See Fig. \ref{intro-1}.)
 \end{definition} 
 
 \begin{remark}\label{rmk:area-estimate}
\begin{enumerate}
	\item It is not hard to show that an annulus $M \in \Cc$ is disjoint from the $z$-axis and, from that, it follows that a curve in $M$ 	is homotopically trivial if and only if its winding number around $Z$ is zero. 

	\item For the annuli  $M\in \Cc$ or for smooth limits of these surfaces we have the following area and curvature estimates \cite{white-entropy,annuloids}: there exists finite 		constants $c_1,c_2$, for which
		\begin{itemize}
			\item $\area(M\cap \BB(p,r))\leq c_1r^2$, 
			\item $|A(M,p)|\min\{1, \dist(p,\partial M),\dist(p, Z)\}\leq c_2.$

		\end{itemize}
\end{enumerate}
where $A$ is the second fundamental form of  $M$ and $\BB(p,r)$ is the ball of radius $r$ in $\RR^3$ centered at  $p$.
 \end{remark}

\begin{definition} \label{curves}
We denote by $\curves$  the space of pairs  $\Gamma=\Gamma_{\rm in} \sqcup \Gamma_{\rm out}$  in $\RR^2\times\{0\}$ of disjoint, Jordan curves in $\RR^2 \times \{0\}$ that are
  nested, symmetric, $C^{2,\alpha}$, convex and have strictly positive curvature.

The map $M\mapsto \partial M$ defines a natural projection

$$\Pi: \Cc \longrightarrow\curves .$$
\end{definition}

  \begin{figure}[htbp]
\begin{center}
\includegraphics[height=5.5cm]{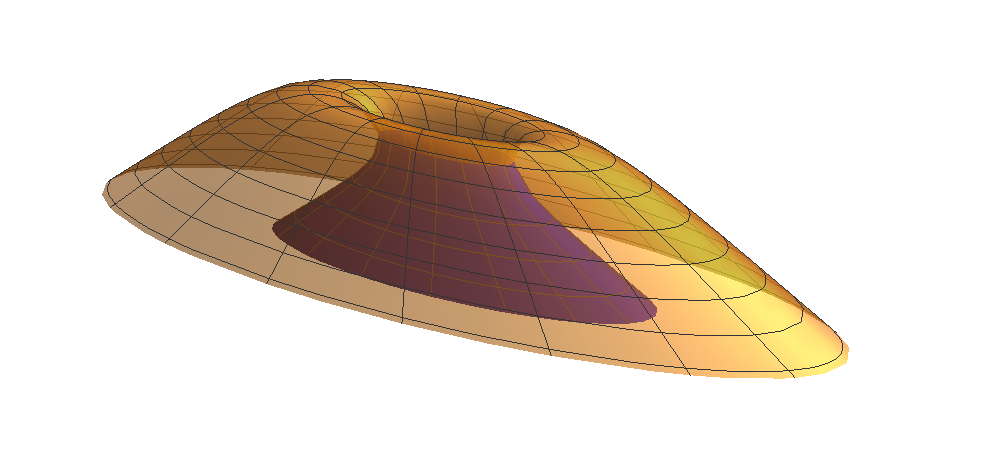}
\caption{\small A translating annulus with compact boundary consisting of two nested, planar, convex curves. \label{fig:compact-annular-translator}}
\label{intro-1}
\end{center}
\end{figure}

At this point, we do not know very much about the space $\Cc$. Indeed, for a pair of curves corresponding to a point  in $\curves$, there may not  be  any annuli in $\Cc$ with that pair of curves as boundary. 
 Looking again at the example discussed in Remark~\ref{rem: R3-gap}, consider a curve $\Gamma$ in the plane  $\{z=0\}$, and let  $\Gamma_t= \Gamma+t\ee_3$  be its translate in the plane  $\{z=t\}$.  Define $\Gamma(t) = \Gamma  \sqcup \Gamma_t$.  For small $t$, there is stable a minimal annulus that is  close to the  vertical ribbon connecting $\Gamma$ to
$ \Gamma_t$. There is also an unstable minimal annulus that looks like  the union of disjoint  planar disks joined by a  small catenoidal neck. For large enough $t$, there is no connected minimal surface whose boundary equals $\Gamma (t)$. 
 But there is a  smooth family $M(t)$, of minimal annuli parametrized by an open interval,with    $\partial M(t)=\Gamma (t)$.  When compactified, we get a closed interval. The annuli going to one endpoint  converge, as sets, to $\Gamma$. At the other end, the convergence is to the disk $D$ bounded by $\Gamma$, and that convergence is smooth, of multiplicity $2$,  away from $\Gamma$ and the point where the catenoid-like handle collapses. Moreover, the length $l(t)$ of the shortest geodesic in $M(t)$, is a continuous function on this interval, which implies that all values in the interval $(0,l(\Gamma))$ 
are assumed.

 We will take the discussion above for minimal annuli  in $\RR^3 $ as a model.

  Fix $\Gamma_0$,  a  smooth,  closed, convex 
  curve in the plane $\{z=0\}$, symmetric with respect to the coordinate axes and with strictly positive curvature. Let
\begin{equation}\label{eq:Gamma(t)}
\Gamma(t)=\Gamma_{\rm in}(t)\sqcup\Gamma_{\rm out}(t),\,\,\, t\in[0,1]
\end{equation}
 be a smooth family of pairs of closed, convex, symmetric curves with the following properties:
 \begin{enumerate}
  \item  For $t\in(0,1)$, $\Gamma(t) $ is a smooth path in $\curves$ 
 \item $\Gamma(0)=\Gamma\sqcup\Gamma$

 \item $\Gamma_{\rm in}(t)$ is a compact curve for all $t\in[0,1]$
 \item $\Gamma_{\rm out}(1)$ is the boundary of a strip $[-d, d]\times \RR$, where  $d$ is large enough so that the distance between $\Gamma_{\rm out}(1)$ and  $\Gamma_{\rm in}(1)$ is at least 
 $\pi$.
 \end{enumerate}
   \begin{figure}[htbp]
\begin{center}
\includegraphics[height=5.5cm]{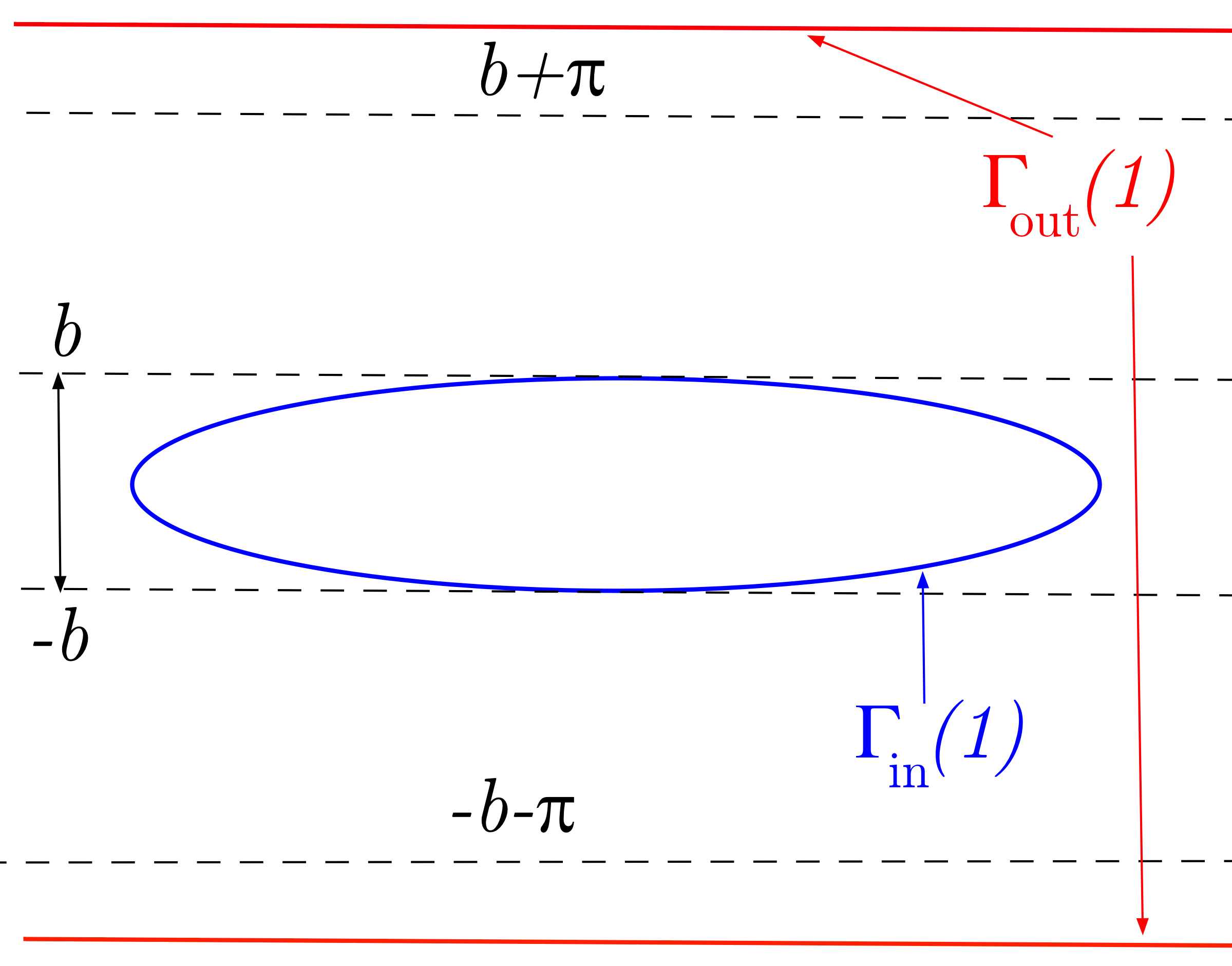}
\caption{\small $\Gamma(1)=\Gamma_{\rm in}(1)\sqcup\Gamma_{\rm out}(1)$, an endpoint of  the path $\Gamma(t)$ in \eqref{eq:Gamma(t)}.}
\label{fig:gamma-t}
\end{center}
\end{figure}

In previous sections we defined $D(\Omega)$ to be the unique translating graph over a convex planar region $\Omega$. with zero boundary values.  In this section 
we will abuse that notation and allow ourselves, for a convex plane curve $\Gamma$,  to denote by
  $D(\Gamma)$ the unique graphical translator with boundary $\Gamma$.  Our goal is to prove that there exists a connected family $\Ff '$ of translating annuli in $\Cc$, each one of which
 has boundary equal to $\Gamma (t)$ for some value $t\in(0,1)$.  Furthermore, when considered as subsets of $\RR^3$ the closure of $\Ff'$ is compact and is equal to $\Ff:=\Ff'\cup\{\Gamma, D\}$. 
 In addition,  the necksize $x(M)$ (see Definition~\ref{x(M)})  is a continuous function that takes on all values in the interval $(0,x(\Gamma))$, where
 $x(\Gamma)$ the distance from the origin to $\Gamma\cap\{y=0\}$.

 We will do this by using the path-lifting theorem, which follows  from White \cite{white87}:
 \begin{theorem}\label{th:lifting} Let $$\mathcal M=\Pi ^{-1}(\Gamma (0,1)).$$ If $\Gamma(t)$ is transverse to $$\Pi:\Cc\rightarrow \curves,$$ then $\mathcal M$ has the structure of a smooth $1$-manifold. More over the map
 $$t:\mathcal M\rightarrow (0,1), \mbox{ where } \partial M=\Gamma (t),$$
is smooth.
 \end{theorem}

 We will use this theorem to find an open, connected component of $\mathcal M$ whose closure consists of the curve with the addition of $\Gamma$ at one end and $D (\Gamma)$ at the other.
  \begin{figure}[htbp]
\begin{center}
\includegraphics[height=4.5cm]{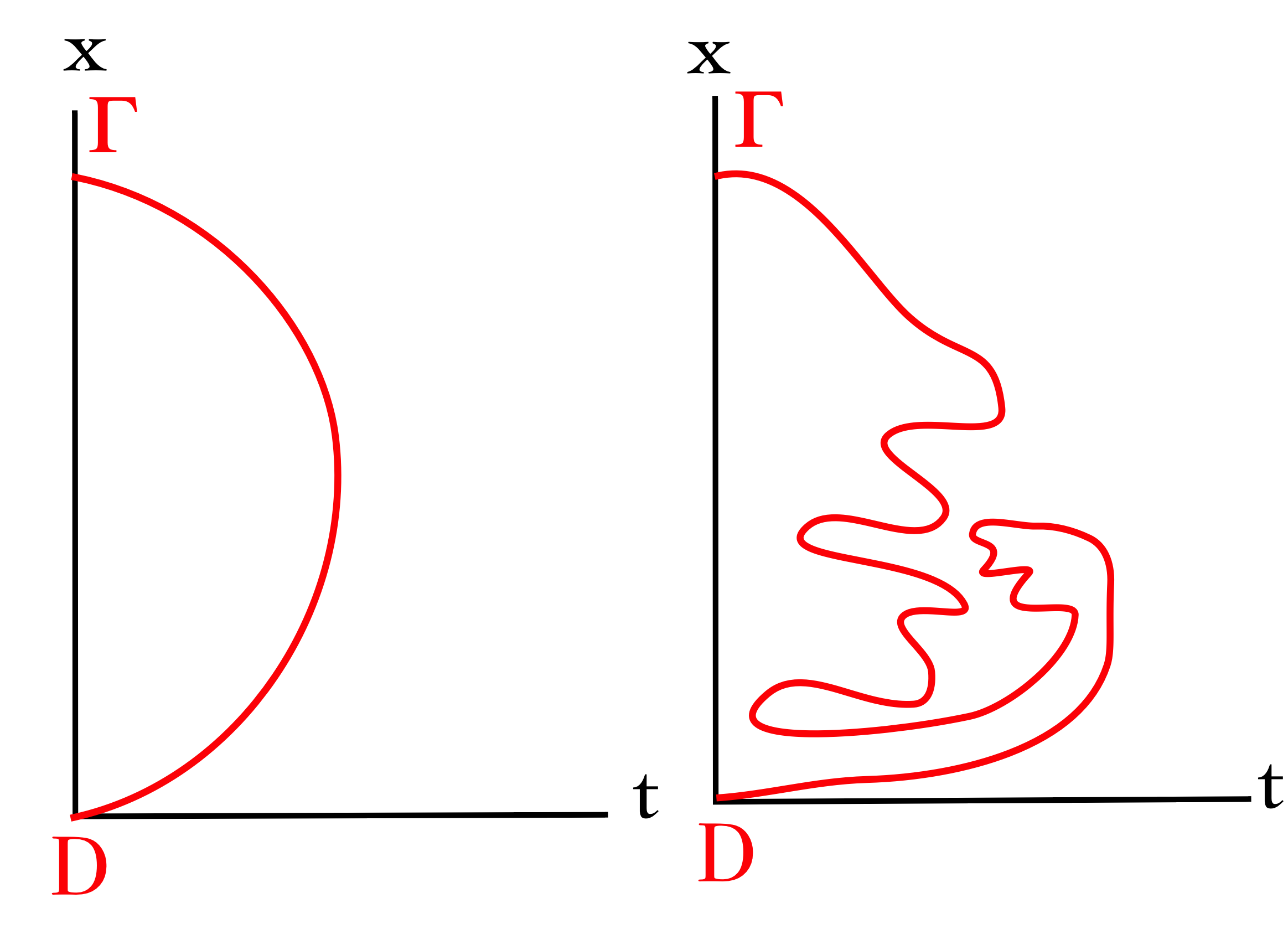}
\caption{\small Left: Lift of a path consisting of pairs of circles $\Gamma\sqcup\Gamma+t\ee_3$, as discussed in Remark~\ref{rem: R3-gap} and before equation \eqref{eq:Gamma(t)}. 
Right:  A fanciful illustration of  $\mathcal M$, the lift of the path $\Gamma_{\rm in} (t)\sqcup\Gamma_{\rm out} (t) $, defined in\eqref{eq:Gamma(t)}. }
\label{fig:spruck-3}
\end{center}
\end{figure}

 \begin{theorem} \label{th:Ff}There exists a connected family $\Ff '$ of  compact translating annuli   with the following properties
 \begin{enumerate}
 \item Each surface $M\in \Ff '$ is a compact annulus in $\Cc$ and there exists a $t=t(M)\in (0,1)$ for which $\partial M=\Gamma (t)$.
 \item $\Ff'=\Ff \cup\{\Gamma, D(\Gamma)\}$ is compact. Specifically, if $M_i$ is a divergent sequence in $\Ff'$ then 
 \begin{enumerate}
 \item Either  the $M_i$ converge to $\Gamma$ as sets, 
 \item Or  the $M_i$ converge to $D(\Gamma)$. The convergence is smooth with multiplicity $2$ on compact subsets of $\RR^3\setminus\{\Gamma, D(\Gamma)\}$.
 \end{enumerate}
 \end{enumerate}

 \vspace{0.1 in}
 
\noindent Moreover, the function $M\rightarrow x(M)$ takes on every value in the interval $(0,x(\Gamma))$.
 \end{theorem}
 \section*{The proof of Theorem~\ref{th:Ff}}

 We will use Theorem~\ref{th:lifting} for the path $\Gamma(t)$ given in \eqref{eq:Gamma(t)}.
   The first step in achieving our goal is to establish properness for the function $t(M)$ defined on $\mathcal {M}$.
\begin{proposition} \label{prop:proper}The map $t:\mathcal M\rightarrow (0,1)$ is proper and bounded above by a constant $c<1$.
\end{proposition}
 \begin{proof}  For the family $\Gamma (t)$ in  \eqref{eq:Gamma(t)},
we imposed condition $(4)$ in order assert, using Theorem~\ref {th:gap3}, that 
 for $t$ sufficiently close to $1$, $\Gamma(t)$ bounds no connected translator in the halfspace $\{z\geq 0\}$. Therefore, there
 is a $T\in (0,1)$ such that if $M$ is a symmetric translating annulus with $\Pi (M) =\Gamma (t)$, then $t<T$. Consequently 
 $$\bigcup_{M\in \mathcal M}\partial M \subset \bigcup_{t\in[0,T]}\Gamma(t)$$
lies in a compact set $K$. By Lemma~\ref{boundedness-lemma} (in Subsection~\ref{sub:gap}), it follows that $\bigcup_{ M\in\mathcal M}M$ also lies in a compact region of $\RR^3$.
Using the curvature and area estimates in Remark~\ref{rmk:area-estimate} (2), it follows that the map  $t:\mathcal M\rightarrow (0,1)$ is proper.
 
 \end{proof}
 
 Theorem~\ref{th:lifting} and Proposition~\ref{prop:proper}  tell us that the only way that a sequence in $\{M_i\}\subset\mathcal M$ with  $\partial M_i =\Gamma(t_i)$ can diverge is if
 $t_i\rightarrow 0$, and therefore the two curves in $\Gamma (t_i)$ are collapsing to $\Gamma$. So we are in a situation similar to that of  catenoidal minimal surfaces in $\RR^3$ discussed after Definition~\ref{curves}.
 
   If $t$ is small, the two plane curves
 $\Gamma (t)$ are very close to $\Gamma$ and the one can show  that there is   a $g$-area-minimizing minimal surface that is a graph over  the flat planar ribbon between the  two curves in $\Gamma(t)$, with  the height  of the graph going to zero with the distance between the two curves. For such a surface $M$, $x(M)$ will be close to the distance from the origin to $\Gamma\cap\{y=0\}$,  a quantity we will denote by  $x(\Gamma)$.  Also one expects that that there will be an (unstable) $g$-minimal annulus with the same boundary, and that annulus will be close  to  union of the two graphical $g$-minimal disks $D(\Gamma_{\rm in}(t))$ and $D(\Gamma_{\rm out}(t)$)),  together with a small neck joining them. And that small neck will be close to the $z$-axis. If $M$ is such a surface, then $x(M)$ is close to zero.  In fact, that is what happens.

\begin{theorem}\label{both} 
Let $M_i\in\Cc$ be a sequence with $\partialin M_i$ and $\partialout M_i$, the boundary components of $M_i$, 
both smooth curves with nonvanishing curvature.  Suppose they converge  in $C^{2,\alpha}$ to the same  smooth convex curve $\Gamma$. 
Denote by $\sigma(M_i)$ the length of the shortest homotopically nontrivial curve in $M_i$.
\begin{enumerate}
\item If $\sigma(M_i)\to 0$, then $M_i$ converges to $D(\Gamma)$, and the convergence is    
  $C^{2,\alpha}$  away from $\Gamma\cup (Z\cap D(\Gamma))$, with multiplicity $2$.
\item If $\inf_i\sigma(M_i)>0$, then $M_i$ converges to $\Gamma$.  For large $i$, 
$M_i$ is the graph of a function $u_i$ on the annular region between
   $\partialin M_i$ and $\partialout M_i$.  
Furthermore, 
\[
   \sup |Du_i| \to 0.
\]
\end{enumerate}
\end{theorem}

From Theorem~\ref{both} we have the following consequence: Suppose $t_i>0$ is a sequence converging to $0$ and $M_i$ is  connected  translator with $\partial M_i =\Gamma(t_i)$. Then after passing to a subsequence,  the $M_i$ converge, as sets, to either $\Gamma$ or $D(\Gamma)$. Moreover,

\begin{enumerate}
\item[i)] If the convergence is to $\Gamma$, then for suitably large, $i$,  $M_i$ is a graph.
\item[ii)] If the convergence is to $D(\Gamma)$, then the convergence is smooth  and with multiplicity $2$ on compact subsets of $\RR^3\setminus\{\Gamma\cup (D(\Gamma)\cap Z)\}$.
\end{enumerate}

\begin{proposition}\label{prop: minimizing-graph}There exists a $\delta>0$ such that if $t\in(0,\delta)$  then a  $g$-area-minimizing  translator with boundary equal to $\Gamma(t)$ is a graph over
the  planar annulus between $\Gamma_{\rm in} (t)$ and $\Gamma_{\rm out} (t)$. Moreover, there is a unique such translating surface $M(t)$ with these properties (graphical or $g$-area-minimizing) and  the $M(t)$ converge as sets to $\Gamma$.
\end{proposition}
\begin{proof} Denote by $\Omega_{out}$ and $\Omega_{in}$ the planar domains bounded by $\Gamma_{\rm out}$ and $\Gamma_{\rm in}$, respectively
Suppose $t_i\rightarrow 0$ and  $M(t_i)$ is a sequence of $g$-area-minimizing  annuli with $\partial M_i=\Gamma(t_i)$. Then
$$\area_g(M(t_i))\leq \area_g(\Omega_{out}\setminus \Omega_{in}).$$
Note that  the right-hand side  of the above inequality  converges to $0$ as $t_i\rightarrow 0$.  Hence,   $\area_g(M(t_i))$ converges to zero. This implies that  the $M_i$ converge
to $\Gamma$ (in Hausdorff distance). If not it would violate the  the monotonicity formula. 
This means we may use statement $(i)$ above to conclude that for $t_i$ suitably small, $M(t_i)$ is a graph. One can now use the maximum principle to show that there is only one graphical translator with the same boundary as $M(t_i)$
\end{proof}

For $0<t<\delta$ denote by $M_g(t)$ the unique graphical translator with boundary equal to $\Gamma(t)$. Then

$$\mathcal E =\bigcup\{M_g(t)\, :\, t\in(0,\delta)\}\subset \mathcal M$$
is an end of a component, $\mathcal F '$, of $\mathcal M$, and that component is a connected $1$-manifold.   Let  $M_i$ be a sequence diverging to the other 
end of $\mathcal F '$. By the properness result above (Proposition~\ref{prop:proper}), we must have $t(M_i)\rightarrow 0$.  For $t(M_i)<\delta$, $M$
cannot be a graph. By Statement $(ii)$, above, the surfaces  $M_i$ have the property that as $t(M_i)\rightarrow 0$, the $ M_i$ converge to the simply connected graph $D(\Gamma)$ 
and the convergence is  smooth, with multiplicity $2$ on compact subsets of $\RR^3\setminus\{\Gamma\cup (D(\Gamma)\cap Z)\}$. Let
 $\mathcal F=\mathcal F'\cup\{\Gamma, D(\Gamma)\}$ and note that from $(ii)$ we have that $x(M_i)\rightarrow 0$. Morover, the neckwidth 
 $x(M)$ is a continuous function on $\mathcal F$, taking the value $0$ at $D(\Gamma)\}$ and $x(\Gamma)$ at $\Gamma$.

This completes the proof of Theorem~\ref{th:Ff}.

\section*{Compact Translating Annuli bounded by Rectangles}

In Definition~\ref{def:Acomp}, if we remove  from condition (1) the requirements that the boundary curves of $M\in\Cc$ are smooth and have strictly positive curvature, we may consider annular translators bounded by nested, symmetric, convex curves. These are boundaries described in  Definition~\ref{curves} of $\curves$ of the surfaces in $\Cc$ with the corresponding smoothness and positive curvature requirements  removed. Can we still assert the conclusions of  Theorem~\ref{th:Ff} for this expanded class?  

Consider a curve of boundaries  \eqref{eq:Gamma(t)}, for $t\in [0,1]$, which are convex but not necessarily  smooth or possessing  strictly positive curvature. Approximate them by 
one-parameter families 

$$t\in[0,1]\rightarrow \Gamma^n(t)=\Gamma^n_{in}(t)\sqcup \Gamma^n_{out}(t),$$
where all the curves in question satisfy the smoothness and positive curvature conditions of Definition~\ref{curves}. We choose these families so that for each $t\in[0,1]$, 
$\Gamma^n_{in}(t)$ converges to $\Gamma_{\rm in}(t)$ and $\Gamma^n_{out}(t)$ converges to $\Gamma_{\rm out}(t)$. We may choose $\Gamma^n(t)$ to be transverse to 
the projection $\Pi$.
Let $\mathcal M^n$, $\Gamma^n$ and $D^n$ be as in Theorem~\ref{th:Ff}.  That theorem asserts the existence of a compact, connected set $\Ff^n$ of closed subsets of $\RR^3$
that have the structure of a compact curve with endpoints $\Gamma^n$ and $D^n$, and interiors points $M$, each of which has boundary 
$\Gamma^n(t)=\Gamma^n_{in}(t)\sqcup \Gamma^n_{out}(t) $ for some $t=t(M)\in(0,1)$.  Passing to a subsequence, the $\Ff^n$ converge to a limit set $\Ff$  which is connected  and contains $\Gamma$ and $D(\Gamma)$. If $M\in\Ff\setminus\{\Gamma, D(\Gamma)\}$, then $M$ is the limit of surfaces in $M^n\in\Ff^n$ . Taking subsequential limits  we may assume that $t\in[0,1]$. By Theorem~\ref{th:gap3}, $t\neq 1$. By Proposition~\ref{prop: minimizing-graph} we can conclude that $t>0$.  
 We can approximate convex, symmetric  curves  by smooth convex curves with strictly positive curvature, so we have proved Theorem~\ref{th:Ff} for such families. There is one class of such curves that is natural and important for our construction.

\begin{definition}\label{Rr-definition}
We define $\Rr$ to be the space of  compact, translating annuli with the following properties:
\begin{enumerate}[\upshape(1)]
\item $\partialin M$ are $\partialout M$ are rectangles whose sides are
  parallel to the coordinate axes, and
\item $M$ is the limit of a sequence of translators $M_n\in \Cc$ 
  such that, for each $n$, $\partialin M_n$ and $\partialout M_n$ are smooth with nowhere vanishing curvature.
\end{enumerate}
\end{definition}

For use in the next section, we state the following theorem, an immediate consequence of the discussion above.

\begin{theorem} \label{th:a,b,x} For $a,b >0$, let $\Gamma_{a,b}$ be the symmetric rectangle that is the boundary of $[-a,a]\times[-b,b]\subset\RR^2$. Then for every $x\in(0,a)$, there
exists a compact annular translator $M_{a,b}(x)$ whose necksize is $x$ and whose boundary consists of  two rectangles $\Gamma_{a,b}$ and $\Gamma_{A,B}$, where
$A\in(a,\infty)$ and $ B\in(b, b+\pi)$.

\end{theorem}
Note that there is no statement of uniqueness for $M_{a,b}(x)$ or for the values $A$ and $B$.

\section{Annuloids  as limits of compact translating annuli with rectangular boundaries}\label{sec:Aa}
 In Section~\ref{sec:DeltaWings}, the $\Delta$-wings  were constructed as limits of sequences of graphical  disks. Here we take limits of sequences of compact annular translators,
$M_{a,b}(x)$, whose existence was established in Theorem~\ref{th:a,b,x}. 

Let  
$$M_i=M_{a_i,b_i}(x_i), \mbox{ with \,} b_i\rightarrow b,
 \,\, x_i\rightarrow\hat x,
  \mbox{\, and \,} a_i\rightarrow\infty.$$ 
Define $z_i=z(M_i)$  to be the largest value of $z$ such that $(x_i,0,z)\in M_i$.

\begin{proposition}\label{prop:Mbx} Let
\[
  M_i':= M_i - (0, 0, z_i).
\]
 A subsequence converges smoothly to a complete translator $M$ that has the following properties  
\begin{enumerate}
 \item $M$ is symmetric with respect to the planes $\{x=0\}$ and $\{y=0\}$.
\item $x(M)=\hat x$, and $(\hat x, 0,0)\in M$.

 \end{enumerate}
 \end{proposition}
\begin{proof}First, we will show that
\begin{equation}\label{pushed-up}
  z_i\to \infty.
\end{equation}

Note that $ (x_i, 0, 0)\in M_i'$.  If \eqref{pushed-up}  fails, 
then (after passing to a subsequence) $z(M_i)$ would 
converge to a finite limit $\hat z$ and $M'_i$ would converge to a translator $M''$ in $\{z\ge 0\}$.(By the curvature and area estimates for these surfaces (Remark~\ref{rmk:area-estimate} (2)), the convergence is smooth in $\{z>0\}$.)
Moreover,  $(\hat x,0,0)\in M''$.  The boundary of $M''$ is contained in the plane $\{z=-\hat z\}$ and lies outside the strip  $\{|y|<b\}$.
 Therefore $M''$ has no boundary in the
slab $\{|y|<b\}$, and we are assuming $b\geq \pi/2$.   By Theorem~\ref{th:gap3}, $M''$ is disjoint from that slab, which is impossible since
$(\hat x, 0, 0)\in M''$.  Thus $z_i\to \infty$.

It follows now that ,
\begin{equation}\label{boundary-away}
  \dist( 0, \partial M_i') \to \infty.
\end{equation}
Consequently, the curvature and area bounds in Remark~\ref{rmk:area-estimate} (2) give smooth convergence (after passing to a subsequence)
of $M_i'$ to a limit translator $M$.  By~\eqref{boundary-away}, $M$ has no boundary.
  From the construction, $(\hat x, 0,0)\in M$.
Also, $M$ is disjoint from the strip $\{0\}\times (-\hat x, \hat x)\times \RR$.  Thus $x(M)=\hat x$. Furthermore, the symmetry of all the surfaces $M_i'$ is inherited by the subsequential limit surface $M$.
\end{proof}

\begin{definition}  \label{def:Aa}

Let $\pi/2\leq b\leq B\leq b+\pi$.  
We define $\Aa (b,B,\hat{ x})$  to be the collection of  limit surfaces $M$ in Proposition~\ref{prop:Mbx} with the property that $b(M_i)\to b$,  $x(M_i)\to \hat x$, and $B(M_i)\to B$ .  (See Theorem~\ref{th:a,b,x}   for the definition of $B(M_i)$). We let
$$\Aa:=\bigcup_{b, B,\hat x} \Aa (b,B,\hat{ x}).$$

\end{definition}

\begin{theorem}\label{th:Aa-as-limit}  
The surfaces in $\Aa$ are annuloids (that is, they satisfy all the conditions of Definition~\ref{def: annuloid})  If $M\in\Aa (b,B, \hat{ x})$, it has inner width $b$, outer width $B$ and necksize $\hat x$.
Moreover:
\begin{enumerate}

\item[i.] $M\cap\{x>\hat x\}$ consists of two simply connected components, $M_{lower}$ and $M_{upper}$.
\item[ii.] $M\cap\{y=0\}\cap\{x\geq0\}$ is a connected curve that consists of the graphs of two functions,   $\phi_{lower}$ and  $\phi_{upper}$, both defined on the interval $[\hat x,\infty)$.
	\begin{itemize}
	\item [$\cdot$] $\phi_{lower}(\hat x)=\phi_{upper}(\hat x)$, and $\phi_{lower}( x)<\phi_{upper} (x)$ for $x>\hat x$.
	\item [$\cdot$] The graph of $\phi_{lower}$ lies in $M_{lower}$ and the graph of  $\phi_{upper}$ lies in $M_{upper}$.
	\end{itemize}
	
\noindent As $x\rightarrow\infty$,
\item[iii.] $M_{lower}-(x,0,\phi_{lower}(x))$ converges to the negatively tilted grim reaper surface defined over the strip $\RR\times (-b,b)$ 
\item[iv.]  $M_{upper}-(x,0,\phi_{upper}(x))$ converges to a  tilted grim reaper surface defined over the strip $\RR\times (-B,B)$. The tilt is positive if $B>b$ and can be either positive or negative if $B=b$.
\end{enumerate}
\end{theorem}

 \noindent Since the  annuloids  $\Aa$ 
 are symmetric with respect to reflection in the coordinate plane, $\{x=0\}$, there are  parallel statements for $M\cap\{x\leq 0\}$.

\vspace{0.2in}

These properties and more are proved in \cite{annuloids}. 
In the remainder of this section, we describe one of the techniques used in establishing that these limit surfaces are  in $\Aa$ and have these properties. For example, the next section indicates how one approaches 
Property~4 in Definition~\ref{def: annuloid}  (that the limit of vertical translates are planes) and properties~$(iii)$ and $(iv)$ of Theorem~\ref{th:Aa-as-limit} above.

\section*{Minimal Foliation Functions}
Let $M$ be a translator.  There are a number of standard foliations $\Ff$ of $\RR^3$  of open subsets
of $W$ of $\RR^3$ by translators.  For example, $\Ff$ could be a family of parallel vertical planes with $W=\RR^3$, or $\Ff$ could be all vertical translates
of a tilted grim reaper surface or of a $\Delta$-wing, with $W$ equal to a vertical slab of width $b\geq \pi/2$.
  Motivated by problems in the study of minimal annuli, we developed  in \cite{morse-rado} general results that allow one to bound the number of points of tangency (counting multiplicity) of $M$ with the leaves of $\Ff$ in terms of the boundary data and the Euler characteristic of $M\cap W$.
    Such bounds can be used to prove  geometric results.

\begin{definition} Let $\Ff$ be a foliation of an open subset $W$ of a Riemannian $3$-manifold $N$ by minimal surfaces. 
For a minimal surface $M\subset N$, a {\bf critical point of $M$ with respect to $\Ff$} is an interior point $p$ where a leaf of $\Ff$ is tangent to $M$
but $M$ is not equal to that leaf in a neighborhood of $p$. The multiplicity of the critical point is the order of contact of $M$ with the leaf. We denote by

\[
   \mathsf{N}(\Ff|M)
\]
 the total number of interior critical points, counting multiplicity.
\end{definition}
\begin{definition}
 A {\bf minimal foliation function} is a continuous function $F$ from an open subset $W\subset N$ to an open interval $I\subset \RR$
such that 
\begin{itemize}
\item $ F^{-1}(t)$ is a minimal surface.
\item $ F^{-1}(t)$ is in the closures of $\{F>t\}$ and $\{F<t\}$.
\end{itemize}
We define  
\[
   \mathsf{N}(F|M)= \mathsf{N}(\Ff|M),
\]
where $\Ff$ is the foliation whose leaves are the level sets of $\Ff$.
\end{definition}

\begin{theorem}[\cite{morse-rado}*{Theorem~4}]
Suppose $F:W\subset N\rightarrow I$ is a proper minimal foliation function on an open domain $W$ of a Riemannian manifold $N$  
and that  $M\subset N$  is a minimal surface with finite genus.  Assume that $(\partial M) \cap \{F < t\}$ is empty for some $ t\in I$.
Let 
\begin{itemize}
\item $Q$ equal the set of local minima of $F|\partial M$ (presumed to be finite).
\item $A $ equal the set of local minima of $F|\partial M$ that are not local minima of $F|M$.
\item $\chi (M\cap W)$ be  the Euler characteristic of $M\cap W$. 
  \end{itemize}
Then
\begin{equation}\label{critical point-est}
\mathsf{N}(F| M)\leq |Q|-|A|-\chi (M\cap W),
\end{equation}
where $|\cdot|$ denotes the number of elements in a set.
\end{theorem}

We want to apply this theorem to the surfaces of Proposition~\ref{prop:Mbx} by computing or estimating  $\mathsf{N}(F| M)$
for finite annuli, and then taking limits. For that we use the following lower-semicontinuity result.
\begin{theorem}\label{semicontinuity-theorem}
Suppose that $M_i$ are minimal surfaces in a $3$-manifold $N$ and that $M_i$ converges smoothly to a minimal surface $M$.
Suppose $F_i$ are minimal foliation functions defined on open subsets $W_i\subset N$ such that that the $F_i$
converge smoothly to a minimal-foliation function $F$ defined on an open subset $W$ of $N$.  Then
\[
   \mathsf{N}(F|M) \le \liminf \mathsf{N}(F_i|M_i).
\]
In particular, if $p$ is a critical point of $(F,M)$, then $p$ is a limit of critical points $p_i$ of $(F_i,M_i)$.
\end{theorem}

\section*{Minimal foliation functions given by vertical planes and by translating graphs}
We will use two   minimal foliation functions. (Recall we are working in $\RR^3$ with the Ilmanen metric $g_{i,j}=e^{-z}\delta_{i,j}$.)

First, if $\vv$ is a horizontal unit vector in $\RR^3$, then
 the function 
\begin{equation}\label{definition-F_v}
\begin{aligned}
&F_\vv: \RR^3\to \RR, \\
&F_\vv(p) = \vv\cdot p
\end{aligned}
\end{equation}
is a minimal foliation function on $\RR^3$.
Second, suppose that $U$ is $\RR^2$ or an open strip in $\RR^2$ and that $h:U\to\RR$ is a function
whose graph is a complete translator.  Then
\begin{equation}\label{general-H}
\begin{aligned}
   &H: U\times \RR \to \RR, \\
   &H(x,y,z) = z - h(x,y)
\end{aligned}
\end{equation}
is a minimal foliation function.

\begin{proposition}\label{prop:N-estimates}The compact translating annuli $M=M_{a,b}(x)$  of Theorem~\ref{th:a,b,x} and the  complete embedded annuloids $M\in\Aa$ of Theorem~\ref{annuloids Aa} that are their limits
satisfy
\begin{enumerate}[\upshape (i)]
\item\label{2-bound-item} For each horizontal unit vector $\vv$, the function $F_\vv$ in \eqref{definition-F_v} satisfies $$\mathsf{N}(F_\vv|M)\le 2.$$
\item\label{4-bound-item} If $H$ is as in~\eqref{general-H}, 
  then $$\mathsf{N}(H|M)\le 8.$$
\end{enumerate} \end{proposition}
\begin{proof} We  prove  the easier case $(i)$ here. For the other case, see Theorem~6.4 in \cite{annuloids}.  A compact annulus $M$ in $\Cc$ is bounded by a pair of strictly convex, nested and symmetric curves. By the definition of strict convexity, 
on each  one of the boundary curves there is precisely one local minimum of $F_\vv$. Since $\chi (M)=0$, it follows from \eqref{critical point-est} that 
$$\mathsf{N}(F_\vv|M)\le 2.$$ We may use the lower-semicontinuity property of Theorem~\ref{semicontinuity-theorem} to assert the same estimate if $M$ is bounded by rectangles as in 
Theorem~\ref{th:a,b,x}. Now taking a limit of compact annular translators with rectangular boundary curves, and using Theorem~\ref{semicontinuity-theorem} again, we get the same estimate for
annuloids that are limits of these sorts or compact annuli, as in Proposition~\ref{prop:Mbx} and Theorem~\ref{th:Aa-as-limit}. (Recall that we have the curvature and area estimates needed in Remark~\ref{rmk:area-estimate}.)

\end{proof}

\begin{theorem} [\cite{annuloids}*{Theorem~5.4}] Suppose $M$ is a properly embedded translator for which we have estimates $(i)$ and $(ii)$ of Proposition~\ref{prop:N-estimates}. and the area and curvature estimates of Remark~\ref{rmk:area-estimate}. Suppose $p_i$ is a sequence in $M$ that diverges in $\RR^3$. Then the sequence of surfaces
$M-p_i$
has a subsequence that converges smoothly to a limit surface $M'$  that is the union of vertical planes and translating graphs. If $p_i=(0,0,z_i)$, then the limit surface is the union of vertical planes.
\end{theorem} 
\begin{proof} We may assume that $M$ is connected and not a vertical plane. By hypothesis $F_\vv$ has no more than two critical points. Let $U$ be the set of critical point of
 $F_\vv|M$ 
and define
$$M_i=(M\setminus U)-p_i.$$
 Since the $p_i$ diverge, the curvature and area estimates of Remark~\ref{rmk:area-estimate} guarantee that the $M_i$ converge smoothly to a limit translator $M'$.
 By lower semicontinuity, 
 $$\mathsf{N}(F_\vv|M')\le\liminf \mathsf{N}(F_\vv|M_i)=\liminf \mathsf{N}(F_\vv|(M'-U))=0.$$
Therefore 
$$ \mathsf{N}(F_\vv|M')=0.$$
Let $\Sigma$ be a component of $M'$ that is not a vertical plane. From the equality above we may conclude that $ \mathsf{N}(F_\vv|\Sigma)=0$  for any horizontal unit vector 
$\vv$. In other words, the tangent plane to $\Sigma$ is never vertical. It now follows from Lemma~\ref{lem:graphical} below that $\Sigma$ is a graph.
(Alternatively, Corollary 1.2 in \cite{spruck-xiao} implies that $\Sigma$ is a graph.)
Now let $W$ be the set of critical points of $H$ on $M$. Recall that the level surfaces of $H$ are all graphs of the same type over the same strip in the plane $\{z=0\}$. For a divergent sequence of points $p_i=(0,0,z_i)$  a parallel argument to the one in the previous paragraph (with $W$ instead of $U$) shows that the sequence of surfaces
 $$M_i=(M\setminus W)-p_i.$$ has a subsequence that converges smoothly to a limit translator $M'$  with 
 $$ \mathsf{N}(F_\vv|M')=0.$$
Suppose that a component $\Sigma$ of $M'$ is a graph. This contradicts Lemma~\ref{lem:N(graph)>0} below, which asserts that for a translating graph $G$, we can find a grim reaper surface or tilted grim reaper surface so that its associated minimal foliation function $H$ satisfies
$$ \mathsf{N}(H|G)\geq 1.$$
\end{proof}

\begin{lemma}\label{lem:graphical} Suppose $\Sigma$ is a connected and properly embedded surface lying in a convex domain of $W\subset\RR^3$ with the property that 
$W\setminus \Sigma$ has two connected components. If the tangent planes to $\Sigma$ are never vertical, then $\Sigma$ is a graph over a horizontal plane.
\end{lemma}

\begin{proof} 
By assumption, the projection
$$\Pi:\Sigma\rightarrow \RR^2=\{z=0\}$$
is locally one-to-one. If it is not globally one-to-one, there exists a vertical line $L$ that intersects $\Sigma$ in two or more points. Choose $p$ and $q$ in $L\cap \Sigma$ with the property
that the interval $l\subset L$ between $q$ and $p$ contains no other points in $L\cap \Sigma$. Relabel the points if necessary so that $q$ is above $p$, i.e. $q\cdot\ee_3>p\cdot\ee_3$.
Now, by assumption,  $W\setminus\Sigma$ has two components. We may orient $\Sigma$ so that its unit normal $\nu$ points into the component that contains the line segment $l$. Therefore, 
$$\nu(p)\cdot\ee_3>0 \mbox{\,\, and\,\,} \nu(q)\cdot\ee_3<0.$$ 
Since $\Sigma$ is connected, there is a path in $\Sigma$ between $p$ and $q$. By the intermediate-value theorem, there must be a point on this path where $\nu\cdot\ee_3=0$.  At this point, the tangent plane to $\Sigma$ is vertical, a contradiction.

\end{proof}

\begin{lemma}\label{lem:N(graph)>0} 
Suppose $G$ is a complete translating graph. Then there is a grim reaper surface for which the associated minimal foliation function $H$ in 
\eqref{general-H} satisfies

$$  \mathsf{N}(H|G)\geq 1.$$
\end{lemma} 
\begin{proof}According to the classification in Theorem~\ref{th:graphs},   we may assume that either $G$ is a bowl soliton or, after a rotation,  $G$ is a grim reaper surface, a tilted grim reaper surface or a $\Delta$-wing defined  over a strip $\{|y|<b\}$ for some value of $b\geq \pi/2$.  Both the $\Delta$-wing  and the bowl soliton contain a  point where the height is maximized. In these two cases, let $H$ be the minimal foliation function associated with the standard grim reaper surface: $h(x,y)=\log(\cos y)$. At these maxima, $H|G$ has a critical point. Therefore 
$ \mathsf{N}(H|G)\geq 1$.

Next, suppose that $G$ is the grim reaper surface. Rotate $G$ by a nonzero angle around the $z$-axis, $Z$, to produce $\tilde G$, and let 
$H$ be the minimal foliation function associated with vertical translates of $\tilde G$. At the origin, $\tilde G$ and $G$ are tangent. Hence $ \mathsf{N}(H|G)\geq 1 $.

The last case to consider is when $G$ is a tilted grim reaper surface. Observe that the Gaussian image of $G$ is a semi-circular arc in the  hemisphere of $S^2\cap\{z<0\}$ with endpoints at  $(0, \pm1,0)$ on the equator. This arc does not pass through the point $(0,0-1)$.  Now rotate the {\em untilted} grim reaper surface by a nonzero angle $\theta$ about $Z$. Its Gaussian image is a semi-circular arc that passes through  $(0,0-1)$ and has endpoints at $\pm(\cos\theta, \sin\theta,0)$ on the equator, for some $\theta\neq 0$. These two semi-circular arcs cross. Therefore, there is a  translate of the this rotated grim reaper surface that is tangent to $G$ at some point. Hence, if $H$ is the minimal foliation function associated with the translated and  rotated grim reaper, we must have, again, $ \mathsf{N}(H|G)\geq 1 $.
\end{proof}
 \section{The proofs of the theorems in Section~\ref{sec:intro}}\label{sec:proofs}
 \section*{A gap theorem for translators}\label{sub:gap}

In the remaining part of this paper we will prove the three theorems in Section~\ref{sec:intro}. In one way or another they depend upon  
 Proposition \ref{prop:gap}.

\begin{proposition} [\cite{annuloids}*{Theorem~18.1}]
\label{prop:gap}
Suppose that $I$ is an infinite open strip in $\RR^2=\{z=0\}$ of width $\pi$.   Suppose that $M$ is a properly immersed translator in $\RR^3\cap\{z\ge 0\}$ with no boundary in the slab $S:=I\times\RR$.  Then $M$ lies in the complement of $S$.
\end{proposition}
 \noindent For the reader's convenience, we provide a proof here.
\begin{proof}

 We may assume that the strip $I$ is $\RR\times(-\pi/2,\pi/2)$.  Fix $\beta\in(0,\pi/2)$. Define the substrip $I_\beta:=\RR\times(-\beta,\beta)\subset I$  and, for $L>0$, the rectangle $R_{L,\beta}:=[-L,L]\times[-\beta,\beta]$. As discussed in Section~\ref{sec:DeltaWings}, there is a unique graphical translator over $R_{L,\beta}$ with zero boundary values. By straightforward applications of the maximum principle, 
\begin{itemize}
\item[(i)] $M\cap (R_{L,\beta}\times\RR)$ lies above the graphical translator over $R_{L,\beta}$.
\item[(ii)] If $L'>L$, the graphical translator over $R_{L',\beta}$ lies above the graphical translator over 
$R_{L,\beta}$.
\end{itemize}
Vertically translate by 
$-\log\cos(\beta)$ the standard grim reaper surface $G$, the graph of $\log(\cos(y))$ over $I$.  Restrict attention to $\{z\geq 0\}$ to produce a graph $G_\beta$ over $I_\beta$.  Note that $G_\beta$  is graph with zero boundary values on $I_\beta$. Using the maximum principle again, 
\begin{itemize}
\item[(iii)]  For any $L>0$, the graph $G_\beta$ lies above the graphical translator  over $R_{L,\beta}$.
\end{itemize}
From (ii) and (iii) above, we may conclude that the graphs over $R_{L,\beta}$ converge smoothly, as $L\to\infty$, to a graphical translator $V_\beta\subset \{z\geq 0\}$ over $I_\beta$. 
This graph is zero on $\partial I_\beta$. Moreover, from (i),
\begin{equation}
\label{Mabove}
M\cap (I_\beta\times\RR)\,\,\mbox{lies above}\,\, V_\beta.
\end{equation}
This is true for every $\beta\in(0,\pi/2)$. We claim that $V_\beta=G_\beta$. Assuming the claim we can conclude the proof of the proposition  from \eqref{Mabove} by
 observing that  a point  $p=(x,y,z)\in S=I\times \RR$ is {\em below} $G_\beta$\ provided $|y|<\beta$ and 
$\beta$ is  sufficiently close to $\pi/2$.

To prove the claim  we will use the following proposition.

\begin{proposition} [\cite{graphs}*{Proposition~3.2}] \label{prop:cylindrical-graph}
Let $\Omega$ be a bounded, convex domain in $\RR^p$.   Suppose there
is bounded translating graph W over $ \RR^q \times\Omega$ 
that vanishes on the boundary.   Then $W$ is unique, and therefore is invariant under translation in the first $q$ coordinates
  translation in the 
depends only on the second  $p$ coordinates.
\end{proposition}
With $p=q=1$ and $\Omega=[-\beta,\beta]$, we  have that  $V_\beta \subset I_\beta\times \RR$  is a horizontal cylinder defined over a curve in $\{x=0, |y|<\beta\}$. 
  In particular,  it has Gauss curvature identically equal to $0$.
As argued in Remark~\ref{rmk:flat-implies GR}, this implies that $V_\beta$ is a portion of a tilted grim reaper surface. Since it contains horizontal lines,
  it follows that  $V_\beta$  is a subset of a translate of the untilted grim reaper surface $G$. Hence $V_\beta =G_\beta$  as claimed.

\end{proof}

\begin{corollary}[\cite{annuloids}*{Corollary~18.3}]\label{convex-corollary}
If $M\subset \{z\ge 0\}$ is a translator, then $M$ lies in $C\times [0,\infty)$, where
$C$ is the convex hull of the projection of $\partial M$ to the horizontal plane.
\end{corollary}

 The following lemma follows easily from Proposition~\ref{prop:gap}  and the maximum principle.

\begin{lemma}[\cite{annuloids}*{Lemma 18.4}] \label{boundedness-lemma}
Let $M \subset \RR^3 \cap \{z\ge 0\}$ be a translator such that $\partial M$  lies in a compact set $K$.  Then $M$ is bounded above. In particular, if $K$ lies below a bowl soliton $Q$, then $M$ also lies below $Q$.
\end{lemma}

  \begin{figure}[htbp]
  
\begin{center}
\includegraphics[height=3.5cm]{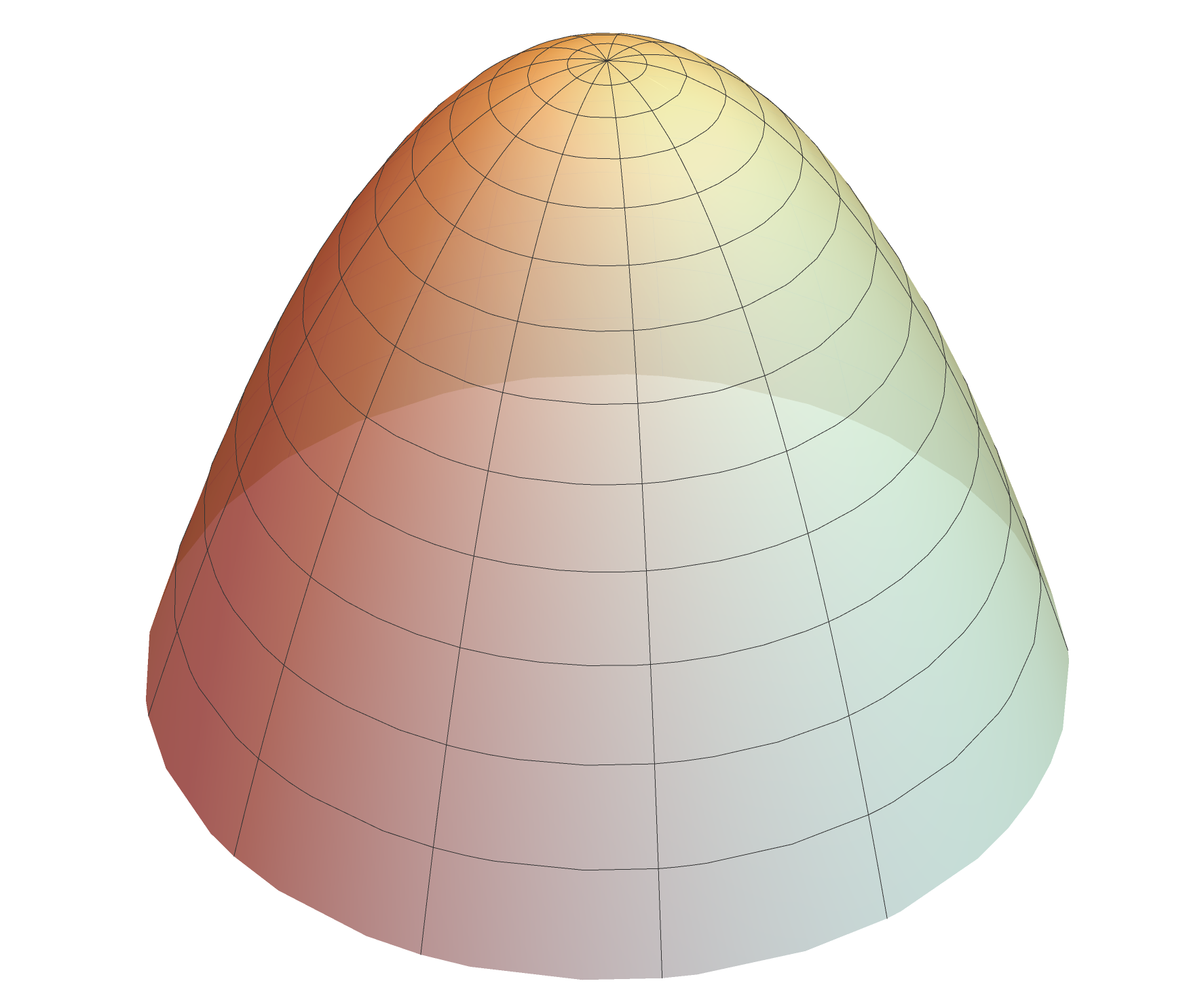} \hfill \includegraphics[height=3.5cm]{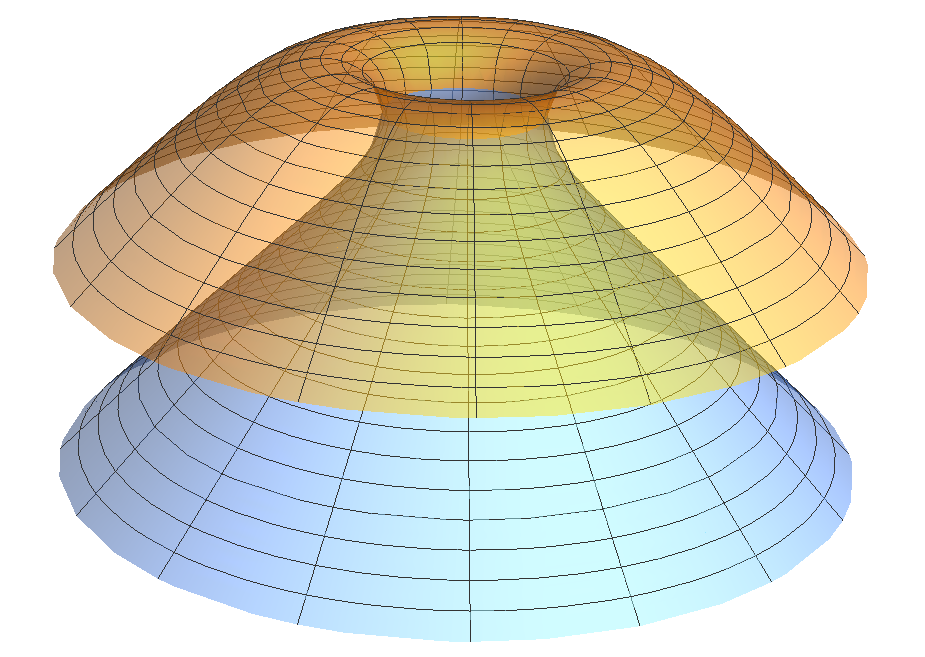}
\caption{\small (Left) A bowl soliton. (Right) A translating catenoid of rotation in the family $W(\lambda)$, defined in Section~\ref{sec:annuloids}.  }
\label{fig:catenoid}
\end{center}
\end{figure}

\section*{Proof of Theorem~\ref{th:gap3} of Section~\ref{sec:intro} } 
 We recall the statement of Theorem~\ref{th:gap3}.

\begin{*theorem} \label{th:gap}
 Let $U_n\subset U_n'$ be nested, open, convex regions in $\RR^2$ such that $U_n$ converges to a bounded, open,
convex set $U$ and such that $U_n'$ converges to an infinite strip $U'$.
Suppose that
\[
    \min\{ |p-q|: p\in \partial U, q\in \partial U'\} \ge \pi.
\]
Then for all sufficiently large $n$, there is no connected translator in $\{z\ge 0\}$ whose
boundary is $S_n:=((\partial U_n)\cup (\partial U_n'))\times\{0\}$.
\end{*theorem}

 \noindent This is Theorem 18.5 in \cite{annuloids}. For the reader's convenience, we provide the proof from that paper.
\begin{proof}
Suppose the result is false.  Then (after passing to a subsequence) each $S_n$ bounds a connected translator $M_n$
in $\{z\ge 0\}$.  Passing to a further subsequence, the $M_n$ converge as sets to a limit set $M$.
Note that $U'\setminus U$ contains two parallel infinite strips $I_1$ and $I_2$ each of width $\pi$.  Thus
by Proposition~\ref{prop:gap}, 
 $M$ is disjoint from $I_1\times\RR$ and $I_2\times \RR$.
Hence, $M$ is the union of three connected components, where one component, $M^*$, has boundary $(\partial U)\times\{0\}$,
and where each of the other two components is bounded by one of the straight lines in 
  $(\partial U')\times \{0\}$.

By Lemma \ref{boundedness-lemma}, $M^*$ is compact.
Let $K$ be a compact set such that $M^*$ is in the interior
of $K$ and such that $M\setminus M^*$ is disjoint from $K$.  For all sufficiently large $n$,
$M_n$ contains a point in $\partial K$, and therefore $M\cap \partial K$ is nonempty, 
a contradiction.
\end{proof}

\section*{Linear bounds on height}

\begin{proposition}\label{prop: quarter-slab}
Given $b\in (0,\infty)$, there is a $\lambda<\infty$ with the following property.
If $M$ is a translator in $\{z\ge 0\}$ with boundary $\partial M$ contained
in 
\[
    [0,\infty)\times [-b,b]\times \{0\},
\]
then $z\le \lambda x$ for all $(x,y,z)\in M$.
\end{proposition}

\begin{proof}[Proof of Proposition~\ref{prop: quarter-slab}]
By Proposition~\ref{prop:gap}, 
\[
  M\subset [0,\infty)\times [-b,b] \times [0,\infty).
\]
Let $B>b$ and consider the rectangular box $[0,B]\times [-B,B]\times [0,B]$.
Let $Q$ be the polyhedral surface consisting of the $x=0$ face and the $y=\pm B$
faces of the box.  Let $w: \RR^2\to \RR$ be a bowl soliton.
By adding a constant to $w$, we can assume that $\partial Q$ lies below the graph of $u$.
Let $S$ be a surface (locally integral current, say) in the region
\[
  K := [0,B]\times [-B,B] \times [0,\infty)
\]
with $\partial S=\partial Q$ that minimizes area with respect to the translator metric.
Then $S$ is a translator that is smooth except perhaps at the corners of $\partial Q$.
By Lemma~\ref{boundedness-lemma}, $S$ has compact support.
By the maximum principle, $S$ is not tangent to $\partial K$ at any non-corner point 
of $\partial S$.

Thus there is an $\eps$ such that $S$ contains the graph of a smooth function
\[
    u: [0,\eps]\times [-b,b] \to \RR.
\]
with $\pdf{u}x>0$ at all points in its domain.  Hence there is a $\lambda>0$ for which
\[
     u(x,y) \le \lambda x.
\]

We claim that $z\le \lambda x$ for all $x\in M$.
Suppose not.  Then there exists  $ \hat{x}\in[0,  \infty)$ with
\[
  \hat x = \inf \{ x: (x,y,z)\in M, \, z>\lambda x\}.
\]
Let
\[
  \hat M = M - (\hat x, 0, \lambda \hat x))\cap \{z\ge 0\}.
\]
Then $\hat M$ satisfies the hypotheses of the Proposition, and
\begin{equation}\label{inf-zero}
  \inf \{ x: \text{$(x,y,z)\in \hat M$ and $z> \lambda x$} \} = 0.
\end{equation}
Furthermore, by Theorem~\ref{prop:gap}, 
\begin{equation}\label{walled-in}
  \hat M \subset [0,\infty)\times [-b,b]\times [0,\infty).
\end{equation}
If $S+(0,0,\zeta)$ intersected $\hat M$ for some $\zeta\ge 0$, then there would be a smallest 
$x\le 0$ such that $S':=S+ (x,0,\zeta)$ intersected $\hat M$, and the strong maximum principle would be violated at the point of contact. 
(It follows from~\eqref{walled-in} that the point of contact would be at an interior point
of $\hat M$ and of $S'$.)
Thus $S+(0,0,z)$ is disjoint from $M$ for all $z\ge 0$.

It follow that $\hat M \cap \{x\le \eps\}$ lies below the graph of $u$ and thus that
\[
   z\le \lambda x \quad \text{for}\quad \text{$(x,y,z)\in \hat M$ with $x\le \eps$}.
\] 
But this contradicts~\eqref{inf-zero}.
\end{proof}

\begin{corollary}\label{quarter-slab-corollary}
If $M$ is a translator in $\{z\ge c\}$ with boundary contained in 
\[
   [a,\infty)\times [-b,b]\times \{c\},
\]
then $z\le c + \lambda(x-a)$ for all $(x,y,z)\in M$.
\end{corollary}

If $M$ is a surface in $\RR^3$, define
\[
  \zmax(M,t) = \sup \{z:  (x,y,z)\in M, \, x=t\}.
\]

\begin{proposition}\label{prop:zmax-theorem}
Given $B$, there is a $C=C_B$ with the following property.
If $M$ is a translator in $M\cap \{x\ge 0\} \cap \{|y|\le B\}$ with $\partial M\subset \{x=0\}$, 
then for $0\le x \le x'$,
\[
   \zmax(M,x') \le \zmax(M,x) + C|x'-x|.
\]
\end{proposition}

\begin{proof}[Proof of Proposition~\ref{prop:zmax-theorem}]  Note that the desired inequality is valid if and only if it is valid for all vertical translates of $M$.
 First, let $x=0$. If $\zmax(M,0)=\infty$, there is nothing to prove. Otherwise, we may assume, without loss of
 generality,  that $\zmax(M,0)=0$. Let $M'=M\cap\{z\geq 0\}$. Note that $M'$ satisfies the assumptions
 of Proposition~\ref{prop: quarter-slab}, so the desired inequality is satisfied with  $C=\lambda$. We now know that
 $M\cap\{x=x'\}$ is bounded above for any value of $x'$. Repeating the argument above for $0<x'\leq x''$ completes the proof.

\end{proof}

 \section*{Translating Scherk-like  graphs}
 
 In the proofs of Theorems~\ref{th:bounded-by-lines} and \ref{th:narrow-slabs}, we will use, as barriers,  graphical translators defined over parallelograms and  bounded by four vertical lines at the corners. These are analogs of the classical Scherk minimal surfaces in $\RR^3$, which exist over rhombi. 
 
 As in the minimal surface case, it suffices to find solutions defined over a parallelogram with
boundary values $+\infty$ on one pair of opposite sides, and $-\infty$ on the other pair. We will state the existence result for  Scherk-like translators (\cite {scherk},  Section~3)
in the special case we will use here:  rectangles $(\alpha=\pi/2$ in the notation of \cite{scherk}). (The general result is for parallelograms with angle $\alpha$ between $0$ and $\pi$ and height $\beta$  between 
$0$ and $\pi$.)
\begin{proposition} \label{prop: Scherk-translator} Fix $\beta \in (0,\pi/2)$ and let $R_L=R_{L,\beta}= (-L,L)\times(-\beta, \beta)$. There exists a length $L(\beta)$ and a translator
$$ v_{\beta}:R_{L(\beta),\beta} \rightarrow \RR$$ 
such that 
\begin{align}\label{eqn:scherk}
v_\beta(\pm L(\beta), y ) &=+\infty    \qquad (|y|<\beta), \\
v_\beta(x, \pm \beta) &= -\infty   \qquad (|x| < \alpha). \nonumber
\end{align}
The graph of $v_\beta$ is bounded by the vertical lines passing through the four corners of $R_{L(\beta),\beta}$.
By adding a constant to $v_\beta$, we may assume that  $v_\beta (0,0) =0$.  Furthermore, as $\beta\rightarrow \pi/2$, 
$L(\beta)\rightarrow\infty$ and  $v_\beta\rightarrow u$, the grim reaper surface
$$u(x,y)=\log(\cos y), \,\, (x,y)\in\RR\times(-\pi/2,\pi/2).$$

\end{proposition}

 \begin{figure}[htbp]\label{fig:Scherk translator}
\begin{center}
\includegraphics[height=5.5cm]{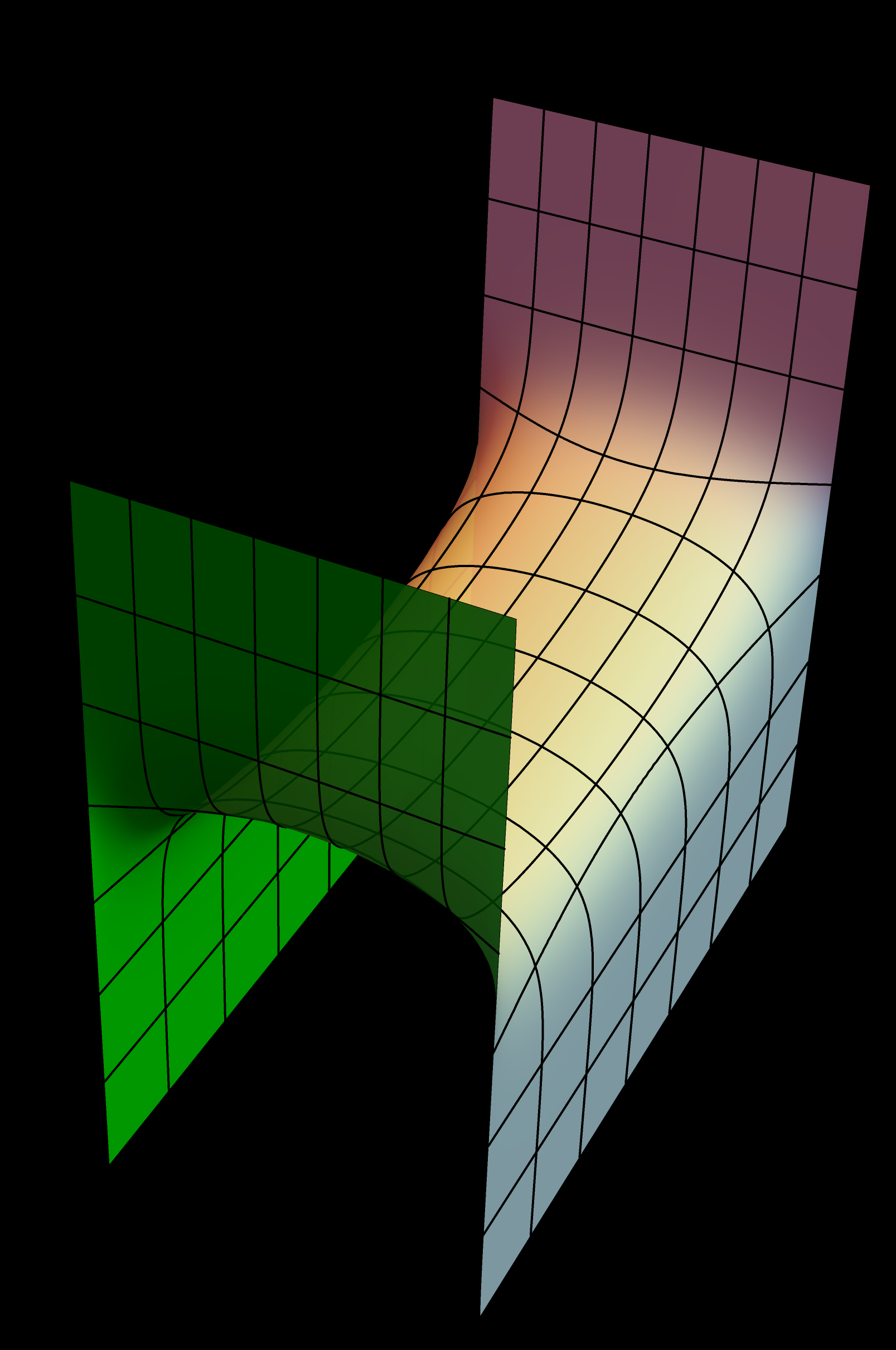}
\caption{\small A Scherk translator, $ v_{\beta}$, as described in  Proposition~\ref{prop: Scherk-translator}.}
\label{fig:scherk}
\end{center}
\end{figure}

\section*{The proof of Theorem~\ref{th:bounded-by-lines}: Translators in   $\{z\geq0\}$ with straight-line boundaries in  $\{z=0\}$}

For the reader's convenience we restate the theorem. Without loss of generality, we may assume that $c=0$.
\begin{*theorem}\label{th:Scherk-translator}
Let $M$ be a connected,  properly embedded translator in $\{z\ge 0\}$ bounded by two parallel lines
$y=\pm b$ in the plane $z=0$.  
Suppose that either
$M\cap \{x=0\}$ is bounded or  that $M$ is simply connected.
    Then $M$ is a portion of a grim reaper surface.
That is, $b<\pi/2$ and $M$ is the graph of the function
$$u-\log(\cos b)= \log(\cos y)-\log(\cos b), \,\,\, |y|<b.$$
\end{*theorem}

\begin{proof}
First we observe (by Proposition~\ref{prop:gap}) that
\[
   M \subset \{ |y|\le b\}
\]
and that $b< \pi/2$.
We will prove the theorem by showing that $M$ lies in the closed region above the graph of $u-\log(\cos b)$ and also in the closed region below the graph of $u-\log(\cos b)$.
\begin{claim}
 If $M\cap \{x=0\}$ is bounded, then there is a $\lambda<\infty$ such that 
\begin{equation}\label{eq: linear bound}
   z\le h + \lambda |x| \quad\text{for all $(x,y,z)\in M$}.
\end{equation}
\end{claim}

\begin{proof}[Proof of Claim 1]Let $h= \sup\{z: (0,y,z)\in M\}$.
By hypothesis, $h< \infty$.  By Proposiiton~\ref{prop: quarter-slab} applied
to $(M-h)\cap\{z\ge 0\} \cap \{x\ge 0\}$ and to $(M-h)\cap\{z\ge 0\}\cap \{x\le 0\}$,
we see that we have the the linear bound \eqref{eq: linear bound}.
 \end{proof}
 
 \begin{claim} If $M$ is simply connected then we have the linear bound \eqref{eq: linear bound}.
 \end{claim}
 \begin{proof}[Proof of Claim 2]
 Because $M$ is simply connected,  a curve $\gamma$ in $M$ that begins on one of the boundary lines and ends on the other
 divides $M$ into two components.   Each component is bounded by $\gamma$ together with two  divergent rays, one 
 in each of the boundary lines. Let  $M_+$ be the component whose  boundary contains rays diverging in the positve $x$-direction,
  and  let $M_-$ be the other one.   Denote by $C_+$  be the convex hull of the projection of $\partial M_+$ onto the strip $\{|y|\leq b\}\cap\{z=0\}$. By 
  Lemma~\ref{boundedness-lemma}, $M_+$  lies in  $C_+\times [0,\infty]$.  For suitably small $a\in \RR$,  $C_+\subset{x\geq a}\times\{|y|\leq b\}\cap\{z=0\}$.
  Hence we may apply Proposiiton~\ref{prop: quarter-slab} to conclude, as in Claim~1, that $M_+$ satisfies  \eqref{eq: linear bound}.  An analogous argument
  works for $M_-$.
 \end{proof}
By Proposition~\ref{prop: Scherk-translator}
for each $\beta$ with $b<\beta< \pi/2$, there is an 
$L=L(\beta)>0$ 
and a function
\[
    v_\beta: R_{L,\beta} \to \RR
\]
such that the graph of $v_\beta$ is a translator satisfying \eqref{eqn:scherk}, with $v_\beta(0,0) =0$.

 By the maximum principle, the minimum value
 of 
 \[
   (x,y,z) \in M \cap \{ |x|< L\} \mapsto  v_\beta(x,y)-z
 \]

 must occur on the boundary of $M \cap \{ |x|< L(\beta)\}$.   Thus,
 \[v_\beta(x,y)-z\ge \min_{|x|\leq L(\beta)} v_\beta(x,b),
 \]
 or
 
  \[v_\beta(x,y) - \min_{|x|\leq L(\beta)}u_\beta(x,b)\geq z.
 \]

 As $\beta \to \pi/2$,  $L(\beta)\rightarrow \infty$,  and the function $v_\beta$  converges to the function $u(x,y)=\log(\cos y)$ on the strip $\RR\times (-b,b)$ between the planes $\{y=\pm b\}$.
 (See Proposition~\ref{prop: Scherk-translator}.)  Hence $M$ lies in the closed region below the graph of $u=\log(\cos y)-\log(\cos b)$. 
  Thus  $M$ lies in the closed region below the graph of $u$.
 
 Now for $a>0$, let $w_a=u_{a,b}$ be the  translator given in \eqref{eqn:D(L,b)}, i.e.
 \[
    w_a: [-a,a]\times[-b,b] \to \RR
 \]
 is the  unique translator with boundary values $0$.  
 
  By the maximum principle, 
 $M\cap \{|x|\le a\}$  lies in the closed region of   $[-a,a]\times[-b,b]\times \RR$ above the graph of $w_a$.  Letting $a\to\infty$, 
 it follows from Proposition~\ref{prop:bounded-limit} and Remark~\ref{rem:b<pi/2} in Section~\ref{sec:DeltaWings} that  $w_a$ converges uniformly to the function $\log(\cos y)-\log(\cos b)$ on $\RR\times[-b,b]$
 Hence,  $M$ lies above the graph of $u$.

 \end{proof}

\section*{The proof of Theorem~\ref{th:narrow-slabs}: Translators in a slab of less than $\pi$}

\begin{theorem*} Suppose $M$ is a properly embedded and connected translator that lies in a vertical slab $\{|y|<B\}$.
If there exists a constant $a$ such that $M\cap\{x=a\}$ is bounded above, then $B\geq \pi$.
\end{theorem*}

 \begin{proof}
 
 Let $M^*=M\cap\{x\geq a\}$. By hypothesis, there  exists a  real number, $c$, such that $M^*:=M\cap\{x\geq a\}\cap\{z\geq c\}$
 satisfies the hypothesis of Corollary~\ref{quarter-slab-corollary}, namely that  $\partial M^*\subset [a,\infty)\times [-B,B]\times\{c\}$. 
 Therefore, 
$$ z\leq c+\lambda (x-a)$$
 for all $(x,y,z)\in M^*$. 
 
Suppose $B<\pi/2$. Translate $M^*$ horizontally, if necessary, so that  that $a=0$.
 For  any $\beta$ satisfying $B<\beta<\pi/2$, we may find a Scherk translator $v_{\beta}$ defined on the rectangle $R_{L,\beta} =(-L,L)\times (-\beta, \beta)$,
  (where $L=L(\beta)$) with the property that  $v_{\beta}$ equals $-\infty$ on the horizontal sides of the boundary of  $R_{L,\beta}$ and $+\infty$ 
  on vertical sides.  Since the height of  $M^*\cap \{0\leq x\leq L\} $ is bounded above by $c+\lambda x$, there exists a constant  $d$ so that  $v_{\beta}+d$
  lies above $M^*\cap \{0\leq x\leq L\} $.  This violates the maximum principle because the boundary of the Scherk translator (see  Proposition~\ref{prop: Scherk-translator})
   consists of the four vertical lines
 through the corners of $R_{L,\beta}$, and $\beta>B$.
 \end{proof}

 \noindent
{\bf Conflict of interest}: The authors had no conflicts of interest.

 \end{document}